\documentclass[smallextended]{svjour3}
\smartqed  
\usepackage{graphicx}
\usepackage{listings}
\usepackage{amsmath, amssymb}
\usepackage{bm}
\usepackage{listings}
\usepackage{mathptmx}
\usepackage[svgnames]{xcolor}

\spnewtheorem*{algorithm}{Algorithm}{\bf}{\rm}

 \numberwithin{equation}{section}

\newcommand{\uu}{\underline{u}}

\newcommand{\uX}{\underline{X}}
\newcommand{\btv}{\bm{TV}}

\newcommand{\bfD}{\bm{D}}

\newcommand{\Exp}{\mathrm{Exp}}

\newcommand{\Dist}{\mathrm{Dist}}
\newcommand{\Diam}{\mathrm{Diam}}
\newcommand{\Lip}{\mathrm{Lip}}

\newcommand{\lfs}{\mathrm{lfs}}

\newcommand{\RR}{\mathbb{R}}

\newcommand{\id}{\,\mathrm{d}}

\newcommand{\relmiddle}[1]{\mathrel{}\middle#1\mathrel{}}
\newcommand{\shrink}{\operatorname{shrink}}

\newcommand{\inner}[3]{\left\langle#1,#2\right\rangle_{#3}}
\newcommand{\argmin}{\mathop{\mathrm{arg~min}}\limits}

\newcommand{\norm}[1]{\left\|#1\right\|}

\newcommand{\pmt}[1]{
    \begin{pmatrix}
        #1
    \end{pmatrix}
}

\begin{document}

\title{A new numerical scheme for discrete constrained total variation flows and its convergence}

\author{Yoshikazu Giga \and Koya Sakakibara \and Kazutoshi Taguchi \and Masaaki Uesaka}

\institute{
Y.~Giga \at
Graduate School of Mathematical Sciences, The University of Tokyo, 3-8-1 Komaba, Meguro-ku, Tokyo 153-8914, Japan\\
\email{labgiga@ms.u-tokyo.ac.jp}
\and
K.~Sakakibara \at
Graduate School of Science, Kyoto University, Kitashirakawa Oiwake-cho, Sakyo-ku, Kyoto 606-8502, Japan;
RIKEN iTHEMS, 2-1 Hirosawa, Wako-shi, Saitama 351-0198, Japan\\
\email{ksakaki@math.kyoto-u.ac.jp}
\and
K.~Taguchi \at 
Graduate School of Mathematical Sciences, The University of Tokyo\\
\email{ktaguchi@ms.u-tokyo.ac.jp}
\and
M.~Uesaka \at
Graduate School of Mathematical Sciences, The University of Tokyo, 3-8-1 Komaba, Meguro-ku, Tokyo 153-8914, Japan;
Arithmer Inc., Minato-Ku, Tokyo 106-6040, Japan\\
\email{muesaka@ms.u-tokyo.ac.jp}
}

\date{Received: date / Accepted: date}

\maketitle

\begin{abstract}
    In this paper, we propose a new numerical scheme for a spatially discrete model of total variation flows whose values are constrained to a Riemannian manifold.
    The difficulty of this problem is that the underlying function space is not convex; hence it is hard to calculate a minimizer of the functional with the manifold constraint even if it exists.
    We overcome this difficulty by ``localization technique''  using the exponential map and prove a finite-time error estimate.  
    Finally, we show a few numerical results for the target manifolds $S^2$ and $SO(3)$.
\end{abstract}

\section{Introduction}
We are concerned with a numerical scheme for solving a spatially discrete constrained total variation flow proposed by Giga and Kobayashi \cite{Giga2003}, which is designated as $(\mathrm{DTVF}_{\mathrm{GK}};u_0)$ and called the discrete Giga--Kobayashi (GK) model in the present paper (see Definition \ref{Def1} below).

A general constrained total variation flow (constrained TV flow for short) for $u \colon \Omega \times [0,T) \to M$ is given as
\begin{equation}
\label{eq:original_system}
(\mathrm{TVF};u_0)
    \left\{
    	\begin{array}{lcl}
    		\displaystyle \frac{\partial u}{\partial t} = -\pi_{u}\left( -\nabla \cdot \left(\frac{\nabla u}{|\nabla u|} \right) \right)  &\mbox{in} &\ \Omega \times (0,T), \\
    		\displaystyle \left( \frac{\nabla u}{|\nabla u|} \right) \cdot \nu^{\Omega} = 0 &\mbox{in} &\  \partial\Omega \times (0,T), \\
    		\displaystyle u|_{t=0} = u_0 & \mbox{on} &\  \Omega,
    	\end{array}
    \right.
\end{equation}
where $\Omega \subset \RR^k\,(k \ge 1)$ is a bounded domain with Lipschitz boundary $\partial \Omega$, $M$ a manifold embedded into $\RR^{\ell}\,(\ell \geq 1)$, $u_0 : \Omega \to M$ an initial datum, $\pi_p$ the orthogonal projection from the tangent space $T_p\RR^{\ell} (=\RR^{\ell})$ to the tangent space $T_pM(\subset \RR^{\ell})$ at $p \in M$, $\nu^{\Omega}$ the outer normal vector of $\partial \Omega$ and $T > 0$.
If $\pi_u$ is absent, $(\mathrm{TVF}; u_0)$ is the standard vectorial total variation flow regarded as the $L^2$-gradient flow of the isotropic total variation of vector-valued maps:
\begin{equation*}
    \btv(u) := \int_{\Omega}|\nabla u|_{\RR^{k\times\ell}}\id x.
\end{equation*}
The introduction of $\pi_u$ means that we impose a restriction on the gradient of total variation so that $u$ always takes value in $M$.
The constrained TV flow is also called the ``$M$-valued TV flow'' or ``1-harmonic map flow''. 
In some literature, the equation \eqref{eq:original_system} is called the constrained TV flow equation or system while the constrained TV flow itself means its solution.
However, in this paper we do not distinguish ``flow'' and ``flow equations''.

The discrete GK model is the constrained gradient flow of the total variation in the space of piecewise constant mappings defined in a given partition of $\Omega$.
Originally, \cite{Giga2003} proposed the one space dimensional model.
Later, the high-dimensional case is studied in \cite{Taguchi2018}.
It is formally a system of ordinary differential equations, but the velocity is singular with respect to its variables.
Moreover, because of the presence of manifold constraints, it is impossible to interpret the problem as a gradient flow of a convex function.
The goal of this paper is to give a time-discrete scheme which is not only practically easy to calculate but also converges to the solution of the discrete GK model.

The readers might be interested whether this (spatially) discrete GK model converges to the constrained TV flow \eqref{eq:original_system} if the space grid tends to zero.
This problem is widely open and is not simple even for unconstrained case as discussed at the end of Section \ref{sec:numerical_schemes}.

\subsection{Applications in science and engineering} 
Constrained TV flows $(\mathrm{TVF};u_0)$ have applications in several fields.
The first application  of the flow of this type appears in \cite{Tang2001}, 
where the authors consider the two-dimensional sphere as  the target manifold $M$ in order to denoise  color images while preserving brightness. 
The system with  the target manifold $M$ of all three-dimensional rotations $SO(3)$, is  an important prototype of the continuum model for time-evolution of grain boundaries in a crystal,
proposed in \cite{Kobayashi2000,Kobayashi2005}.
Moreover, equation \eqref{eq:original_system}, where  the target manifold $M$ is the space of all symmetric positive definite three-dimensional matrices $SPD(3)$, is  proposed for denoising MR diffusion tensor image  (\cite{Basser1994,Pennec2006,Christiansen2007,Weinmann2014}).

\subsection{Mathematical analysis} 
Despite its applicability, the mathematical analysis of  the manifold-constrained TV flows is still developing.
Two difficulties lie in mathematical analysis:
One is the singularity of the system when  $\nabla u$ vanishes; the other  is the constraint with values of flows in  a manifold.
Many  studies can be found to overcome the first difficulty.
In order to explain the second difficulty, we distinguish two types of solutions:  ``regular solution'' and ``irregular solution''.

We mean by ``regular solution'' a solution without jumps.
In \cite{Giga2005}, the existence of local-in-time regular solution was proved when $\Omega$ is the  $k$-torus $T^{k}$, the  manifold $M$ is an $(\ell-1)$-sphere $S^{\ell-1}$ and the  initial datum $u_0$ is sufficiently smooth and of small total variation.
Recently, this work has been improved significantly in \cite{Giacomelli2017}.
In particular, the assumption has been weakened to convex domain $\Omega$ and Lipschitz continuous initial data $u_0$.
Moreover,  in \cite{Giacomelli2017}, the existence of global-in-time regular solution and its uniqueness have been proved when the target manifold $M$ has non-positive curvature, and the initial datum  $u_0$ is small.

In \cite{GigaKuroda2004}, it has been proved that rotationally symmetric solutions may break down, that is, lose their  smoothness in finite time
when $\Omega$ is the two-dimensional unit disk and $M=S^2$.
Subsequently, in \cite{Giacomelli2010},
the optimal blowup criterion for the initial datum given in \cite{GigaKuroda2004} was found, and it was proved that the so-called reverse bubbling blowup might happen.

For ``irregular solution'', a solution that may have jumps, two notions of a solution  are proposed, depending on the choice of the distance to measure jumps of the function.
These choices lead to different notions of a solution of a constrained gradient system, which may not coincide.

Weak  solutions derived from ``extrinsic distance'' or ``ambient distance'', the distance of the Euclidean space in which the manifold $M$ is embedded, were studied in \cite{Giga2003,Giga2015}.
In \cite{Giga2003}, the  existence and the uniqueness of global-in-time weak solution in the space of piecewise constant functions,  have been established when $M$ is compact,
the domain $\Omega$ is an interval with finite length,
and the initial datum $u_0$ is piecewise constant.
Moreover, a finite-time stopping phenomenon  of $S^1$-valued TV flows was also proved.
On the contrary, for $S^2$-valued TV flows,
an example that does not stop in finite time was constructed in \cite{Giga2015},
which will be reproduced numerically by our new numerical scheme
and used for its numerical verification in this paper.

Weak solutions derived from ``intrinsic distance'',
the geodesic distance of the target manifold $M$,
were  studied in \cite{Giacomelli2013,Giacomelli2014,Giacomelli2016}.
In \cite{Giacomelli2013}, the existence and the uniqueness of global-in-time weak solution have been proved
when $\Omega$ is a bounded domain with Lipschitz boundary, $M=S^1$,
and the initial datum $u_0$ has finite total variation
and does not have jumps greater than $\pi$.
These  arguments and results were extended in \cite{Giacomelli2016} 
when the target manifold $M$ is a planar  curve.
As for higher  dimensional target manifolds,
the existence of global-in-time weak solution was proved in \cite{Giacomelli2014}
when the target manifold $M$ is a hyperoctant $S^{\ell-1}_+$ of the  $(\ell-1)$-sphere.

If one takes the intrinsic distance, the uniqueness of solution may fail to hold as pointed out in \cite{Taguchi2018}.
This is one reason why we adopt the extrinsic distance.

\subsection{Numerical analysis and computation} 
Discrete constrained TV flows, that is, spatially discrete models of constrained TV flows, have been studied in \cite{Osher2002,Giga2003,Giga2006,Feng2008,Feng2009,Taguchi2018}.

In \cite{Osher2002}, discrete models of $S^{1}$-valued TV flows and $S^{2}$-valued TV flows based on the finite difference method were studied.
More precisely, a  numerical scheme was proposed, and numerical computations were performed. 

In \cite{Feng2008,Feng2009}, $S^{\ell-1}$-valued regularized TV flows based on the finite element method were studied.
In these  works, the existence and uniqueness of global-in-time solution of the discrete models were established, and numerical computations were performed.
We remark that the convergence of the discrete model to the original model was also studied.
However, its argument has some flaws, which were pointed out  in \cite{Giacomelli2013}.

In  \cite{Giga2003,Giga2005,Giga2006,Taguchi2018}, discrete models of constrained TV flows on the space of piecewise constant functions were studied.
Such  discrete models can be directly applied to denoising of manifold-valued digital images.
For  one-dimensional spatial domains, solutions of the discrete models coincide with irregular solutions of the corresponding original model derived by ambient distance but it is not the  case for  higher dimensional domainst.
In one-dimensional spatial domain, the existence and the uniqueness of global-in-time solution to the discrete models were established in \cite{Giga2003}.
Moreover, numerical computations of $S^1$-valued discrete models were performed.
These discrete models are formulated as ordinary differential inclusions; i.e.,  differential equation with multi-valued velocity.
There are two key ideas in \cite{Giga2003} to solve them.
The first  one is computation of the canonical restriction of the multi-valued velocity.
The second  one is to use facet-preserving property  of flows.
We emphasize  that these two key ideas do not work in dimensions higher  than one.
In the higher dimensional  case, the existence and the uniqueness of global-in-time solution of the discrete model were established in \cite{Giga2005,Giga2006,Taguchi2018}.
Numerical computations of these discrete models, however, were not performed. 

\subsection{Contribution of this paper} 
This paper is dedicated to the  study of a numerical scheme for simulation of the discrete model studied  in \cite{Giga2005,Giga2006,Taguchi2018}, of constrained TV flow. 
In particular, we propose a new numerical scheme and show its convergence.
We also perform numerical simulations based on the proposed scheme. 
We outline the three main contributions below:

\subsubsection{New numerical scheme} 
Constrained discrete TV flow is formulated as a gradient flow in a suitable manifold.
Hence, one can use the minimizing movement scheme (see \cite{Ambrosio2005}) to simulate it.
It is summarized as follows: 
Let $H := (H, \langle \cdot, \cdot \rangle_H)$ be a real Hilbert space, $E$ be a submanifold of $H$, $\mathcal{F}$ be a  real-valued functional on $H$ allowing the value $+\infty$, $I := [0,T]$ be a time interval and $\tau>0$ be a step size.
We first consider a sequence $\{u_\tau^{(n)}\}$ in $E$ generated by a simple minimizing movement scheme for $\mathcal{F}$, which will be referred as (MM; $\tau$, $u_0$) in Section \ref{sec:numerical_schemes}.
In this scheme, $\{u_\tau^{(n)}\}$ is determined successively by taking a minimizer in $E$ of a functional
\[
	\tau\mathcal{F}(u)+\frac{1}{2}\|u-u_\tau^{(n-1)}\|_H^2.
\]
Note that the uniqueness of minimizers is not guaranteed since $E$ is not a convex constraint.
In general, if $\mathcal{F}$ were geodesically convex and coercive on $E$, then the piecewise linear interpolation (Rothe interpolation) of $\{u_{\tau}^{(n)}\}$ would converge to the gradient flow of $\mathcal{F}$ (see \cite{Ambrosio2005}) and this general theory could be applied.
However, this is not the case in our situation.
In this simple scheme, we need to solve an  optimization problem at each step.
Each optimization problem is classified as Riemannian optimization problem, i.e., an  optimization problem with Riemanninan manifold constraint.
Theory of \emph{smooth} Riemannian optimization, that is, Riemannian optimization problem whose objective function is smooth is well-studied, and we refer to the systematically  summarized book \cite{Absil2008}. 
However, our energy $\mathcal{F}$ is the total variation energy so it is not smooth.
Unfortunately, \emph{non-smooth}  Riemannian optimization is only  scarcely explored as a subfield of the theory of Riemannian optimization.
We refer to \cite{Osher2014,Wang2015,Kovnatsky2016,Zhang2017} as its references.
Moreover, the theory of non-smooth Riemannian optimization is generalized to the theory of non-smooth and non-convex optimization with a separable structure in \cite{Li2015}.
Although there are several studies in this direction, it is under development as compared with the linearly constrained problem.
For these reasons in this paper, we propose a new minimizing movement scheme which includes only a linearly constrained optimization problem, which is referred as $(\mathrm{MM}_{loc};\tau,u_0)$ in Section \ref{sec:numerical_schemes}.
In this scheme, instead of minimizing
\[
	\tau\mathcal{F}(u)+\frac{1}{2}\|u-u_\tau^{(n-1)}\|_H^2
\]
in $E$, we minimize
\[
	\tau\mathcal{F}(u_\tau^{(n-1)}+X)+\frac{1}{2}\|X\|_H^2
\]
defined for $X$ which is an element of the tangent space $T_{u_\tau^{(n-1)}}E$ contained in $H$.
This is a linearly constrained (convex) optimization problem, which will be referred as $(\mathrm{VP}_{loc};u_\tau^{(n-1)})$.
There is a unique minimizer.
Let $X_\tau^{(n-1)}$ be the minimizer.
We then determine $u_\tau^{(n)}$ by the exponential map of $X_\tau^{(n-1)}$ at $u_\tau^{(n-1)}$.

\subsubsection{Stability and convergence } 
In this paper, we study stability and convergence  of the scheme from  two points of view.
Those are ``discrete  energy dissipation" and ``error estimate" for the proposed scheme.
Since discrete TV flows have the structure of gradient flows, the proposed scheme should inherit the properties of the gradient flow.
Therefore, in this paper, we show that if $\tau$ is sufficiently small, the proposed scheme satisfies the energy dissipation inequality, which is one of the properties of the gradient flow.

Moreover,  we prove that the sequence $\{ u_{\tau}^{(n)}\}$ generated by the modified minimizing movement scheme converges to the original  gradient flow as $\tau \to 0$. 
The  convergence result of the minimizing movement scheme appears in~\cite{Ambrosio2005,Crandall1971,Rulla1996}.
In the classical work~\cite{Crandall1971} on  the convergence analysis in Banach space,
error estimate of order $O(\sqrt{\tau})$  was obtained, and it is improved to the order $O(\tau)$  in~\cite{Rulla1996}. 
A  similar estimate to~\cite{Crandall1971} in general metric space is shown in~\cite{Ambrosio2005}.
Since our scheme contains the ``localization'' process, we cannot apply these previous works to our scheme directly.
We can show, however, the error estimate locally in time between the piecewise linear interpolation, the  so-called Rothe  interpolation, of $\{u_{\tau}^{(n)}\}$,  and the original flow.
This estimate, which will be presented  in Section~\ref{sec:convergence}, states that the error of the Rothe interpolation is $O(\sqrt{\tau})$ as $\tau \to 0$, and corresponds  to those in~\cite{Ambrosio2005,Crandall1971}.
We remark that the idea  using the exponential map appears in the optimization problem in matrix manifolds and in numerical computation of regularizing flows like constrained heat flows (see~\cite{Absil2008,Faugeras2002,Faugeras2004}).
We are concerned, however, with the approximation of the gradient flow in addition to converging to a  minimizer.  
As far as we know, there is no previous rigorous result on the convergence of the  flow itself.

\subsubsection{Numerical simulations} 
The proposed scheme is not enough to simulate constrained discrete TV flow since we need to solve the linearly constrained convex but non-smooth optimization problem $(\mathrm{VP}_{loc};u^{(n-1)}_{\tau})$ at each step.
We overcome this situation by rewriting $(\mathrm{VP}_{loc}; u^{(n-1)}_{\tau})$ as an iteration, and adopt alternating split Bregman iteration, proposed by~\cite{Goldstein2009}, which is adequate to the optimization problem including total variation.
We refer to~\cite{Oberman2011,Pozar2018} for examples of the application of this iteration to calculate the mean curvature flow numerically.
One can find a proof of convergence of this iteration in \cite{Setzer2009}.

In this paper, we provide a numerical analysis of  three aspects  of constrained TV flows.
The first one is a property of $S^2$-valued TV flow, discovered in \cite{Giga2015}, of not stopping in finite time. 
The second one is an error estimate of the proposed scheme; actually,  the example in \cite{Giga2015} can be rewritten as a simple  ordinary differential equation, which can be accurately solved by an explicit scheme.  
Therefore, we can confirm the theoretical convergence rate by comparing the results obtained by the two numerical methods. 
The third one is the  numerical observation that a facet is preserved  in most of evolution.
We simulate $S^2$-valued TV flows with one spatial dimension constructed  in \cite{Giga2015},  and $SO(3)$-valued TV flows in  two spatial dimensions. 

\subsubsection{Advantages} 
The proposed scheme has five advantages:

First, this scheme does not restrict the target manifold $M$.
In many previous studies, the target manifold $M$ is fixed in advance, to the  sphere,  for example.
Our method, however, can be applied to  any Riemannian manifold as target manifold $M$.

Second, this scheme is not restricted to constrained TV flows and can be applied to more general constrained gradient flows.
Indeed, as we  can see from the scheme $(\mathrm{MM}_{loc}; \tau, u_0)$, it is possible to construct a numerical solution $\{ u_{\tau}^{(n)} \}$ if we can solve the linearly constrained problem $(\mathrm{VP}_{loc}; u^{(n)}_{\tau})$ for all $n \geq 0$.

Third, the proposed scheme can describe facet-preserving phenomena of constrained TV  flows.
In the numerical calculation of the TV  flow, we should pay attention to whether a numerical scheme can adequately simulate the evolution of facets. 
Many  schemes that have been proposed so far cannot capture  this phenomenon since the energies are smoothly regularized.
On the other hand, our  scheme is capable of preserving facets  since the energies are only convexified, not regularized.  

Fourth, the proposed scheme is numerically practical.
Especially, if the exponential map and the orthogonal projection $\pi$ of the target manifold $M$ can be calculated easily, the practical advantage of our scheme is clear.
If $M$ is in the  class of orthogonal Stiefel manifolds, its orthogonal projection $\pi$ and its exponential map can be written explicitly.
See also \cite{Absil2008}.
Besides, as mentioned above, our method does not use the projection onto  the target manifold, which is sometimes hard to calculate.

Finally, the proposed scheme is well-defined, and we shall prove its convergence together with convergence rate.
 
\subsection{Organization of this paper} 
The plan of this paper is as follows:

In Section~\ref{sec:preliminaries}, we recall notion and notations for describing constrained discrete TV flows we study in this paper.
More precisely, we recall the notion of manifolds and define a space of piecewise constant functions, discrete total variation, and a discrete model of constrained TV flows
.

In Section~\ref{sec:numerical_schemes}, we propose a new numerical scheme and provide its analysis. 
Namely, in  Section~\ref{sec:derivation}, 
we derive  a new numerical scheme for constrained discrete TV flows, starting  from the minimizing movement scheme for constrained discrete TV flows.
In Section~\ref{sec:rothe_interpolation}, we explain Rothe interpolation.
This interpolation is useful for establishing the energy dissipation inequality and error estimate of the proposed scheme.
In Section~\ref{sec:energy_decay}, we prove that the proposed scheme satisfies energy dissipation inequality if the step size is sufficiently small.
In Section~\ref{sec:convergence}, we establish an error estimate of the proposed scheme.
We see that the error estimate implies that the proposed scheme converges to constrained discrete TV flows as step size tends to zero.
The key of the  proof is to establish the evolutionary  variational inequality, used in \cite{Ambrosio2005}, for Rothe interpolation of the proposed scheme.

In Section~\ref{sec:numerical}, we propose a practical algorithm to perform numerical computations.
More precisely, we rewrite the scheme into a practical version with alternating split Bregman iteration, and we use it to simulate $S^2$-valued and $SO(3)$-valued discrete TV flows in Section \ref{sec:numerical_results}.
In  Appendix~\ref{sec:the_constant}, we explain why  the constant $C_M$ defined by \eqref{eq:Curv+LSF} in Section~\ref{sec:preliminaries} is bounded by explicit quantities associated to the  submanifold $M$.

\section*{Acknowledgement}
We are grateful to Professor~Takeshi~Ohtsuka (Gumma University) for several discussions and comments on split Bregman iteration. 
Notably, he suggested to us that the alternating split Bregman iteration is suitable for the numerical experiments. 
We are also grateful to Professor Hideko Sekiguchi (The University of Tokyo) for her comments on properties of $SO(3)$. 
The authors are grateful to anonymous referees for careful reading and several constructive comments.  

The work of the second author was initiated when he was a postdoc fellow at The University of Tokyo during 2017. 
Its hospitality is gratefully acknowledged.
Much of the work of the third author was done when he was a graduate student at The University of Tokyo. 
The present paper is based on his thesis.
The work of the fourth author was initiated when he was an assistant professor at Hokkaido University during 2017 and 2018. 
Its hospitality is gratefully acknowledged.

The work of the first author was partly supported by Japan Society for the Promotion of Science (JSPS) through the grants KAKENHI No. 26220702, No. 19H00639, No. 18H05323, No. 17H01091, No. 16H0948 and by Arithmer, Inc. through collaborative research grant.
The work of the second author was partly supported by JSPS through the grant KAKENHI No. 18K13455.
The work of the third author was partly supported by the Leading Graduate Program ``Frontiers of Mathematical Sciences and Physics", JSPS.

\section{Preliminaries}\label{sec:preliminaries} 
Here and henceforth, we fix the bounded domain $\Omega$ in $\RR^k$ with Lipschitz boundary and denote by $\nu_\Omega$ the outer normal vector of $\partial\Omega$.

\subsection{Notion and notations  of submanifolds  in Euclidean space} 
Let $M$ be a $C^2$-submanifold in $\RR^\ell (\ell \ge 2)$.
For $p \in M$, we denote by $T_pM$ and to $T^{\bot}_pM$ the tangent and the normal spaces of $M$ at $p$, respectively, and write $\pi_p$ and $\pi^{\bot}_p$ for the the orthogonal projections from $\RR^{\ell}$ to $T_pM$ and $T^{\bot}_pM$, respectively.
We denote by $\mathrm{II}_p : T_pM \times T_pM \to T^{\bot}_pM$ the second fundamental form at $p$ in $M$.
Moreover, define the diameter $\Diam(M)$ and the curvature $\mathrm{Curv}(M)$ of $M$ by 
\begin{equation*}
	\Diam (M) := \sup_{p,q \in M} \norm{p-q}_{\RR^{\ell}},
	\quad
	\mathrm{Curv}(M) := \sup_{p\in M}\sup_{X \in T_pM} \frac{\mathrm{II}_p(X, X)}{\| X \|^2_{\RR^{\ell}}},
\end{equation*}
respectively. 
If $M$ is compact, then $\Diam (M)$ and  $\mathrm{Curv}(M)$ are finite.
Given a point $p \in M$ and velocity $V \in T_pM$, we consider the ordinary differential equation in $\RR^{\ell}$, the so-called geodesic equation in $M$, for $\gamma : [0,\infty) \to \RR^{\ell}$:
\begin{equation}\label{eq:sff}
    \frac{\id^2 \gamma}{\id t^2}(t) - \mathrm{II}_{\gamma(t)}\left( \frac{\id \gamma}{\id t}(t),\frac{\id \gamma}{\id t}(t) \right) = 0, 
    \quad \gamma(0) = p,
    \quad \frac{\id \gamma}{\id t}(0) = V.
\end{equation}
If $M$ is compact, then the Hopf--Rinow theorem implies that there exists a unique curve $\gamma^{p,V} : [0,\infty) \to M$ satisfying \eqref{eq:sff}. 
The curve $\gamma^{p,V}$ is called geodesic with initial point $p$ and initial velocity $V$.
Then the exponential map $\exp_p : T_pM \to M$ at $p$ is defined by the formula $\exp_p(V) := \gamma^{p,V}(1)$.
We have  $\gamma^{p,V}(t) = \exp_p(tV)$ for $t \in [0,\infty)$ since $s \mapsto \gamma^{p,tV}(s/t)$ satisfies the geodesic  equation \eqref{eq:sff} and the solution of \eqref{eq:sff} is unique.  
Moreover, set
\begin{equation}\label{eq:Curv+LSF}
	C_M := \sup_{p, q \in M}\frac{\norm{\pi^{\bot}_p(p-q)}_{\RR^{\ell}}}{\norm{p - q}_{\RR^{\ell}}^2}.
\end{equation}
If $M$ is path-connected and compact, then the constant $C_M$ is finite.
In fact, as explained in Appendix \ref{sec:the_constant}, an explicit bound for $C_M$ can be obtained. 

\subsection{Finite-dimensional Hilbert spaces} 

First, we define partitions with rectangles. 
Let $\Delta$ be a finite set of indices and let $\Omega$ be a bounded domain in $\mathbb{R}^k$. 
A family $\Omega_{\Delta} := \{\Omega_{\alpha} \}_{\alpha \in \Delta}$ of subsets of $\Omega$ is a \emph{rectangular partition}  of $\Omega$ if $\Omega_{\Delta}$ satisfies 
        \begin{enumerate}
        \renewcommand{\labelenumi}{(\arabic{enumi})}
        \item $\mathcal{L}^{k}(\Omega \setminus \displaystyle \bigcup_{\alpha \in \Delta} \Omega_{\alpha}) = 0$; 
        \item $\mathcal{L}^{k}(\Omega_{\alpha} \cap \Omega_{\beta}) = 0$
        for $\alpha \not= \beta$,
        $(\alpha, \beta) \in \Delta \times \Delta$;
        here $\mathcal{L}^k$ denotes the Lebesgue measure.
        \item For each $\alpha \in \Delta$, 
        there exists a rectangular domain $R_{\alpha}$ in $\RR^k$ 
        such that $\Omega_{\alpha} = R_{\alpha} \cap \Omega$. 
        Here, by a rectangular domain, we mean that 
        \[
        R_\alpha = \left\{ x=(x_1,\ldots,x_k) \mid a_j<x_j<b_j,\ \text{for}\ j=1,\ldots,k \right\}
        \]
for some $a_j<b_j$.
    \end{enumerate}
Although our theory works for rather a general partition, we consider only the rectangular partition above for practical use.
    
We set 
\begin{equation}
    v(\Omega_{\Delta}) := \inf_{\alpha \in \Delta} \mathcal{L}^{k}(\Omega_{\alpha}).
    \label{eq:inf_sup_volume}
\end{equation}
We denote by $e(\Delta)$ the set of edges associated with $\Delta$ defined as
\begin{equation*}
   e(\Delta) := \{ \gamma := \{\alpha, \beta \} \subset \Delta \mid \mathcal{H}^{k-1}(\overline{\partial \Omega_{\alpha} \cap \partial\Omega_{\beta}}) \not= 0,\  \alpha \not= \beta \},
\end{equation*}
where $\mathcal{H}^m$ denotes the $m$-dimensional Hausdorff measure. 
For $\gamma := \{\alpha, \beta \}\in e(\Delta)$, we take a bijection $\mathrm{Sign_{\gamma}}: \gamma \to \{\pm 1 \}$ and set $E_{\gamma} := \overline{\partial \Omega_{\alpha} \cap \partial \Omega_{\beta}}$.
Moreover, we define the set $E\Omega_{\Delta}$ of interior edges in $\Omega$ associated with $\Omega_{\Delta}$ and the symbol  $\mathrm{Sign}_{e(\Delta)}$ on $e(\Delta)$ by 
\begin{equation*}
    E\Omega_{\Delta} := \{ E_{\gamma} \}_{\gamma \in  e(\Delta)}, \quad \mathrm{Sign}_{e(\Delta)} := \{ \mathrm{Sign_{\gamma}} \}_{\gamma \in e(\Delta)},
\end{equation*}
respectively.

Subsequently, assume that the partition $\Omega_{\Delta}$ of $\Omega$ is given.
Then, set 
\begin{equation*}
    H_{\Delta} := \left\{ U \in L^2(\Omega;\RR^{\ell}) \relmiddle|\mbox{ $\left.U\right|_{\Omega_{\alpha}}$ is constant in $\mathbb{R}^\ell$ for each $\Omega_{\alpha} \in \Omega_{\Delta}$}  \right\}.
\end{equation*}
We regard the space $H_{\Delta}$ as a closed subspace of the Hilbert space $L^2(\Omega;\RR^{\ell})$ endowed with the inner product defined as
\begin{equation*}
        \inner{X}{Y}{H_{\Delta}} := \langle X, Y \rangle_{L^2(\Omega; \RR^{\ell})}.  
\end{equation*}
For $U \in H_{\Delta}$, we denote the facet of $U$ by  
    \begin{equation*}
        \mathrm{Facet}(U) := \{ \{\alpha, \beta\}  \in e(\Delta) \mid U^{\alpha} = U^{\beta} \},\ \text{where}\ U^\alpha=\left.U\right|_{\Omega_{\alpha}}. 
    \end{equation*}
Further, we  set
\begin{equation*}
    M_{\Delta} := \left\{ u \in L^2(\Omega;M) \relmiddle|  \mbox{  $\left.u\right|_{\Omega_{\alpha}}$ is constant in $M$ for each  $\Omega_{\alpha} \in \Omega_{\Delta}$} \right\}.
\end{equation*}
We regard the non-convex space $M_{\Delta}$ as a submanifold of $H_{\Delta}$
in which the function takes values in  $M$.
Subsequently, for $u \in M_{\Delta}$, we denote by $T_u M_{\Delta}$ the tangent space of $M_{\Delta}$ at $u$, i.e., 
\begin{equation*}
    T_u M_{\Delta} := \left\{ X \in L^2(\Omega;\RR^{\ell}) \relmiddle| \mbox{$\left.X\right|_{\Omega_{\alpha}}$ is constant in  $T_{\left.u\right|_{\Omega_{\alpha}}}M$ for each $\Omega_{\alpha} \in \Omega_{\Delta}$} \right\}.
\end{equation*}
Moreover, we denote by $H_{E\Omega_{\Delta}}$ the space of piecewise constant $\RR^{\ell}$-valued maps on $\bigcup E\Omega_{\Delta}=\bigcup_{\gamma\in e(\Delta)} E_\gamma $ such that 
 \begin{equation*}
    H_{E\Omega_{\Delta}} := \left\{ U \in L^2(\bigcup E\Omega_{\Delta}; \RR^{\ell}) \mid  U|_{E_{\gamma}} \ \mbox{is constant in} \ \RR^{\ell} \mbox{ for each } E_{\gamma} \in E\Omega_{\Delta} \right\}.
\end{equation*}

\subsection{Discrete constrained TV flows } 
The problem $(\mathrm{TVF};u_0)$ is formally regarded as the gradient system of (isotropic) total variation:
\begin{equation*}
\begin{split}
    \btv(u) 
    &:= \int_{\Omega} |\bm{D}u| \\
    &:=  \sup_{\bm{\varphi} \in \mathcal{A}} \ \sum_{j = 1}^{\ell} \int_{\Omega} u^{j} (\nabla \cdot \varphi^{j}) \id x, \quad u := (u^1, \ldots, u^{\ell}) \in L^1(\Omega;\RR^{\ell}),
    \end{split}
\end{equation*}
where
\begin{equation*}
    \mathcal{A} 
    := \left\{ 
            \bm{\varphi} := (\varphi^1, \ldots, \varphi^{\ell}) \in C^{\infty}_0(\Omega; \RR^{k\times \ell}) \relmiddle| \| \bm{\varphi} \|_{\RR^{k \times \ell}} \leq 1
            \right\}.
\end{equation*}
The spatially discrete problems we consider in this paper are regarded  as the gradient system of \textit{discrete (isotropic) total variation}.
Let us begin with the definition of discrete total variation associated with $\Omega_{\Delta}$.
Let $\Omega_{\Delta} := \{ \Omega_{\alpha} \}_{\alpha \in \Delta}$ be a rectangle partition of $\Omega$.
Then the discrete total variation functional $\btv_{\Delta} \colon H_{\Delta} \to \RR$ associated with $\Omega_{\Delta}$ is defined as follows:
\begin{equation}\label{eq:discretized-TV}
    \btv_{\Delta}(u) 
    := \sum_{\gamma \in e(\Delta)} \| (\bfD_{\Delta} u)^{\gamma} \|_{\RR^{\ell}}\mathcal{H}^{k-1}(E_{\gamma}),
\end{equation}
where $\bfD_{\Delta} (:= \bfD_{\Delta}^{\mathrm{Sign}_{e(\Delta)}})  : H_{\Delta} \to H_{E\Omega_{\Delta}}$ is the discrete gradient associated with $\Omega_{\Delta}$ which is for $u:=\sum_{\alpha\in\Delta}u^\alpha\mathbf{1}_{\Omega_\alpha}$ defined by
\begin{equation*}
\begin{split} 
    \bfD_{\Delta} u
    &:= \sum_{\{\alpha, \beta \} \in e(\Delta) } (\bfD_{\Delta} u)^{\{\alpha, \beta \}}\mathbf{1}_{E_{\{\alpha, \beta \}}} \\
    &:= \sum_{\{\alpha, \beta \} \in e(\Delta) } ( \mathrm{Sign_{\{\alpha, \beta \}}}(\alpha)u^{\alpha} +  \mathrm{Sign_{\{\alpha, \beta \}}}(\beta) u^{\beta} )\mathbf{1}_{E_{\{\alpha, \beta \}}}.
\end{split}
\end{equation*}
This definition is easily deduced from the original definition of $\btv(u)$ when $u$ is a piecewise constant function associated with $\Omega_\Delta$.
We remark that the functional $\btv_{\Delta}$ is convex on $H_{\Delta}$ but not differentiable at a point $u$ whose facet $\mathrm{Facet}(u)$ is not empty.
The next proposition is an  immediate conclusion of the  definition of $\btv_{\Delta}$.
Hence, we state it without proof:

\begin{proposition}\label{property:dtv}
    The following statements hold:
    \begin{enumerate}
        \item $\btv_{\Delta}$ is a semi-norm in $H_{\Delta}$.
        \item $\btv_{\Delta}$ is Lipschitz continuous on $H_{\Delta}$, that is, 
        \begin{equation*}
            \mathrm{Lip}(\btv_{\Delta}) := \sup_{u, v \in H_{\Delta}} \frac{| \btv_{\Delta}(u) - \btv_{\Delta}(v) |}{\| u - v \|_{H_{\Delta}}} < \infty.
        \end{equation*}
        \item $\btv_{\Delta}(u)$ is equal to $\btv(u)$ for all $u \in H_{\Delta}$.
   \end{enumerate}
\end{proposition}

Constrained discrete TV flow is  the constrained $L^2$-gradient flow of spatially discrete total variation.
The gradient $\nabla_{M_{\Delta}}\btv_{\Delta}$ of constrained discrete TV flow is, formally, given by
\begin{equation}
    \nabla_{M_{\Delta}}\btv_{\Delta}(u) := P_u \nabla \btv_{\Delta}(u), \quad u \in M_{\Delta}.
\end{equation}
Here, $\nabla \btv_{\Delta}$ denotes a formal gradient in $H_{\Delta}$ and $P_u$ denotes the orthogonal projection from $H_{\Delta}$ to $T_u M_{\Delta}$ naturally induced from $\pi_p$:
\begin{gather}
    P_u X(x) := \pi_{u(x)}(X(x)) \quad \mbox{for a.e. } x \in \Omega.
\end{gather}
The discrete total variation $\btv_{\Delta}$, however, is not differentiable.
Therefore, we use the subgradient $\partial \btv_{\Delta}$ in $H_{\Delta}$ instead of the gradient $\nabla \btv_{\Delta}$:
\begin{equation*}
    \partial \btv_{\Delta}(u) := \{ \zeta \in H_{\Delta} \mid \langle \zeta, v -u\rangle_{H_{\Delta}} + \btv_{\Delta}(u) \leq \btv_{\Delta}(v) \mbox{  for all  } v \in H_{\Delta} \}, \quad u \in H_{\Delta}.
\end{equation*}
Accordingly, the spatially discrete problem we consider is as follows:
\begin{definition} \label{Def1}
    Let $u_0 \in M_{\Delta}$ and $I := [0,T)$.
    A map $u \in W^{1,2}(I;M_{\Delta})$ is said to be \emph{a solution to the discrete Giga--Kobayashi model (GK model for short)  of  $(\mathrm{TVF};u_0)$} if $u$ satisfies
	\begin{equation}
		(\mathrm{DTVF}_{\mathrm{GK}}; u_0)
   		\left\{
		\begin{array}{ll}
			\displaystyle \frac{du}{dt}(t) \in -P_{u(t)}\partial \btv_{\Delta}(u(t)) &\mbox{for a.e. $t \in (0, T)$}, \\
			\displaystyle \left.u\right|_{t=0} = u_0. &
		\end{array}
		\right.
	\end{equation}
\end{definition}

Here, we note that the existence and the uniqueness of global-in-time solution of discrete GK model have been proved in \cite{Giga2003,Giga2005,Taguchi2018}:

\begin{proposition}[\cite{Giga2003,Giga2005,Taguchi2018}]
    Let $u_0 \in M_{\Delta}$, $I := [0,T)$ and $M$ be a $C^2$-compact submanifold in $\RR^{\ell}$.
    Then, there exits a solution $u \in W^{1,2}(0,T; M)$ to the discrete GK model of $(\textrm{TVF}; u_0)$.
    Moreover, assuming that $M$ is path-connected, $u$ is the unique solution.
\end{proposition}

\section{Numerical scheme}\label{sec:numerical_schemes} 

\subsection{Derivation}\label{sec:derivation} 
We seek a suitable time discretization of $(\mathrm{DTVF}_{\mathrm{GK}};u_0)$, so that we fix the step size $\tau > 0$ and denote by $N(I, \tau)$ the maximal number of iterations, that is, the minimal integer greater than $T/\tau$.
Moreover, we define the time nodal points  
\begin{align}
	I(\tau):=\{t^{(n)}\mid n=0,\ldots,N(I,\tau)\},
	\label{eq:time_nodal_point}
\end{align}
where
\begin{equation*}
    t^{(n)} 
    := \left\{
    \begin{array}{ll}
        n \tau &\mbox{if $n = 0, \dots N(I,\tau)-1$,} \\
        T & \mbox{if $n = N(I, {\tau})$.}
    \end{array}
    \right.
\end{equation*}
We focus on the gradient structure of $(\mathrm{DTVF}_{\mathrm{GK}};u_0)$ to use the minimizing movement scheme in \cite{Ambrosio2005}.
Then we obtain the following numerical scheme of $(\mathrm{DTVF}_{\mathrm{GK}};u_0)$:

\begin{algorithm}[$(\textrm{MM}; \tau, u_0)$: Minimizing Movement Scheme]
    Let $u_0 \in M_{\Delta}$ and $\tau>0$ be a step size.
    Then, we define the sequence $\{u_{\tau}^{(n)} \}_{n=0}^{N(I, \tau)}$ in $M_{\Delta}$ by the following scheme:
    \begin{enumerate}
        \item For $n=0$, $u^{(0)}_{\tau} := u_0$.
        \item For $n\geq 1$, $u_{\tau}^{(n)}$ is a minimizer of the optimization problem $(\mathrm{VP}; u_{\tau}^{(n-1)})$:
        \begin{equation*}
            \mbox{Minimize} \   \Phi^{\tau}(u;u_{\tau}^{(n-1)}) \  \mbox{subject to} \  u \in M_{\Delta},
        \end{equation*}
    \end{enumerate}
    where
   \begin{equation*}
        \Phi^{\tau}(u; u_{\tau}^{(n-1)})
        := \displaystyle \tau \btv_{\Delta}(u) + \frac{1}{2}\|u - u_{\tau}^{(n-1)} \|^2_{H_{\Delta}}, \quad u \in H_{\Delta}.
    \end{equation*}
\end{algorithm}

We determine $u^{(n)}_{\tau}$ by the optimization problem $(\mathrm{VP}; u_{\tau}^{(n-1)})$ when we use the minimizing movement scheme.
However, in general  it is not easy to solve the problem $(\mathrm{VP}; u_{\tau}^{(n-1)})$ since  each of  the problems  $(\mathrm{VP}; u_{\tau}^{(n-1)})$ is classified as \emph{non-smooth Riemanian constraint optimization problem}.
This difficulty motivates us to replace $(\mathrm{VP}; u_{\tau}^{(n-1)})$ with an optimization problem easier to handle.
Our strategy is to determine $u^{(n)}_{\tau}$ from the  tangent vector $X \in T_{u_\tau^{(n-1)}} M_\Delta$ which is the optimizer of \emph{non-smooth (convex) optimization problem $(\mathrm{VP}_{loc}; u^{(n-1)}_{\tau})$ with constraint in  the tangent space $T_{u_\tau^{(n-1)}}M_{\Delta}$} and the exponential map $\Exp_{u_{\tau}^{(n-1)}}: T_{u_{\tau}^{(n-1)}}M_\Delta \to M_{\Delta}$.

We explain this idea in more detail.
First, we rewrite the optimization problem $(\mathrm{VP};
u_\tau^{(n-1)})$ to obtain the one with a constraint into the tangent space $X \in T_{u_{\tau}^{(n-1)}}M_{\Delta}$.
Each $u \in M_{\Delta}$, thanks to the exponential map in $M$, can be rewritten as the pair of $u_{\tau}^{(n-1)}$ and $X \in T_{u_{\tau}^{(n-1)}}M_{\Delta}$ such that $u = \Exp_{u_{\tau}^{(n-1)}}(X)$, where $\Exp_{u_{\tau}^{(n-1)}} : T_{u_{\tau}^{(n-1)}}M_{\Delta} \to M_{\Delta}$ is defined as
\begin{equation*}
     X(x) \mapsto  \exp_{u_\tau^{(n-1)}(x)}(X(x)), \quad \mbox{for a.e. } x \in \Omega.
\end{equation*}
Here $\exp_x$ is the exponential map of the Riemannian manifold $M$.  
Since $\Exp_{u_{\tau}^{(n-1)}}(X) = u_{\tau}^{(n-1)} + X + o(X)$, we ignore the term $o(X)$ and insert $u = u_{\tau}^{(n-1)} + X$ into $\Phi^\tau(u, u_{\tau}^{(n-1)})$ in $(\mathrm{VP}; u^{(n-1)}_{\tau})$ to obtain
\begin{equation*}
	\Phi^\tau(u_{\tau}^{(n-1)} + X; u_{\tau}^{(n-1)}) = \tau \btv_{\Delta}(u_{\tau}^{(n-1)} + X) + \frac{1}{2}\norm{X}^2_{H_{\Delta}}.
\end{equation*}
Now, we define \textit{the localized energy} $\Phi^{\tau}_{loc}(\cdot ; u^{(n-1)}_{\tau}): H_{\Delta} \to \RR$ of $\Phi^{\tau}(\cdot ; u^{(n-1)}_{\tau})$ by the formula:
\begin{equation}
	\Phi_{loc}^{\tau}(X; u_{\tau}^{(n-1)}) 
	:= \tau \btv_{\Delta}(u_{\tau}^{(n-1)} + X) + \frac{1}{2}\norm{X}_{H_{\Delta}}^2, \quad X \in H_{\Delta}.
\end{equation}
Here, we emphasize that $\Phi_{loc}^{\tau}(\cdot ; u_{\tau}^{(n-1)})$ is convex in $H_{\Delta}$ since $\btv_{\Delta}$ is convex in $H_{\Delta}$.
Subsequently, we consider the optimization problem $(\mathrm{VP}_{loc}; u_{\tau}^{(n-1)})$:
\begin{equation*}
    \mbox{Minimize} \  \Phi^{\tau}_{loc}(X;u_{\tau}^{(n-1)}) \  \mbox{subject to} \ X \in T_{u_{\tau}^{(n-1)}}M_{\Delta}.
\end{equation*}
The problem $(\mathrm{VP}_{loc};u_{\tau}^{(n-1)})$ is convex because of convexity of $\Phi^{\tau}_{loc}$ in the Hilbert space $H_{\Delta}$ and linearity of the space $T_{u_{\tau}^{(n-1)}}M_{\Delta}$.
Hence, we can find a unique minimizer $X_{\tau}^{(n-1)}$ of $(\mathrm{VP}_{loc}; u_{\tau}^{(n-1)})$.
Finally, we associate $X_{\tau}^{(n-1)}$ with an element of $M_{\Delta}$ by the formula: $u_{\tau}^{(n)} := \Exp_{u_{\tau}^{(n-1)}}(X_{\tau}^{(n-1)})$.
Summarizing the above arguments, we have the following modified minimizing movement scheme $(\mathrm{MM}_{loc}; \tau, u_0)$:

\begin{algorithm}[$(\textrm{MM}_{loc}; \tau, u_0)$: Modified Minimizing Movement Scheme]
    Let $u_0 \in M_{\Delta}$ and $\tau>0$ be a step size.
    Then, we define the sequence $\{u_{\tau}^{(n)} \}_{n=0}^{N(I, \tau)}$ in $M_{\Delta}$ by the following procedure:
    \begin{enumerate}
        \item For $n=0$: $u_{\tau}^{(0)} := u_{0}$.
        \item For $n\geq1$: $u_{\tau}^{(n)}$ is defined by the following steps:
        \begin{enumerate}
            \item Find the minimizer $X_{\tau}^{(n-1)}$ of the variational problem $(\mathrm{VP}_{loc}; u_{\tau}^{(n-1)})$:
            \begin{equation*}
                \mbox{Minimize} \ \Phi_{loc}^{\tau}(X; u_{\tau}^{(n-1)}) \ \mbox{subject to} \ X \in T_{u_{\tau}^{(n-1)}}M_{\Delta},
            \end{equation*}
            where
            \begin{equation*}               \Phi_{loc}^{\tau}(X;u_{\tau}^{(n-1)}) := \tau \btv_{\Delta}(u_{\tau}^{(n-1)} + X) + \dfrac{1}{2}\norm{X}^2_{H_{\Delta}},
                \quad
                X \in T_{u_{\tau}^{(n-1)}} M_{\Delta}.
            \end{equation*}
            \item Set $u_{\tau}^{(n)} := \Exp_{u_{\tau}^{(n-1)}}(X_{\tau}^{(n-1)})$.
        \end{enumerate}
    \end{enumerate}
\end{algorithm}

\begin{remark}[Why the proposed scheme describes  facet-preserving phenomena] 
In the numerical calculation of (constrained) TV flows, we should pay attention to whether the scheme can adequately simulate the evolution of facets. 
To see why the scheme has the facet-preserving property, note that given $u \in M_{\Delta}$ and $X \in T_uM_{\Delta}$, the total variation of $u + X$ is decomposed into
    \begin{align*}
        \btv_{\Delta}(u + X) 
        &= \sum_{\gamma\in e(\Delta) \setminus \mathrm{Facet}(u)} \| (\bm{D}_{\Delta}(u + X))^{\gamma}  \|_{\RR^{\ell}}\mathcal{H}^{k-1}(E_{\gamma}) \\
        &\hspace{70pt}+ \sum_{\gamma \in \mathrm{Facet}(u)} \| (\bm{D}_{\Delta}X)^{\gamma} \|_{\RR^{\ell}}\mathcal{H}^{k-1}(E_{\gamma}).
    \end{align*}
Hence, the minimizer $X_* \in T_uM_{\Delta}$ of the optimization problem 
    \begin{equation*}
        \mbox{Minimize} \ \tau \btv_{\Delta}(u + X) + \frac{1}{2}\norm{X}^2_{H_{\Delta}} \ \mbox{subject to} \   X \in T_{u}M_{\Delta},
    \end{equation*}
tends to have the same facet as $u$, that is, $\mathrm{Facet}(u) = \mathrm{Facet}(X_*)$.
Therefore, $\Exp_u(X_*)$ also tends to have the same facet as $u$.
\end{remark}

This scheme is always well-defined since the minimizer is unique.
Here is the statement:
\begin{proposition}[Well-definedness of the proposed scheme]\label{proposition:scheme} 
	Let $\tau >0$ and $u_0 \in M_{\Delta}$.
    Then, the sequences $\{ u_{\tau}^{(n)} \}_{n =0}^{N(I, \tau)}$, $\{ X_{\tau}^{(n)} \}_{n = 0}^{N(I, \tau)-1}$ in the above modified minimizing movement  scheme are well-defined, and satisfy
	\begin{equation*}
		 X_{\tau}^{(n)} \in -\tau P_{u_{\tau}^{(n)}}\partial\btv_{\Delta}(u_{\tau}^{(n)} + X_{\tau}^{(n)}), \quad 
		 \| X_{\tau}^{(n)} \|_{H_{\Delta}} \leq \tau \ \Lip(\btv_{\Delta})
	\end{equation*}
	for all $n = 0, \ldots, N(I, \tau)-1$.
\end{proposition}

\subsection{Stability and convergence} 
\subsubsection{Rothe interpolation}\label{sec:rothe_interpolation} 
We consider the Rothe interpolation of sequences generated by $(\mathrm{MM}_{loc}; \tau, u_0)$.
This interpolation is useful to prove energy dissipation in Proposition \ref{thm:maindecay} and error estimate in Theorem~\ref{thm:mainerror}.
Set the time interpolation functions $\{ \ell_{\tau}^{(n)}\}_{n = 0}^{N(I, \tau)-1}, \ell_{\tau} \colon I \to [0, 1]$ as follows:
\begin{equation}
	\displaystyle \ell^{(n)}_{\tau}(t) := \frac{t-t^{(n)}}{\tau}1_{[t^{(n)},t^{(n+1)})}(t), \quad \ell_{\tau}(t) := \sum_{n = 0}^{N(I, \tau)-1}\ell^{(n)}_\tau(t),
	\quad
	t\in I.
\end{equation}
\begin{definition}[Rothe Interpolation] 
    Fix the initial datum $u_0 \in M_{\Delta}$.
    Let $\{ u_{\tau}^{(n)} \}_{n=0}^{N(I, \tau)}$ be the sequence generated by the modified minimizing movement scheme $(\mathrm{MM}_{loc};\tau, u_0)$.
    Then, we define two interpolations $\uu_{\tau},u_\tau \colon I \to M_{\Delta}$ as follows:
    \begin{align*}
        \uu_{\tau}(t) &:= \sum_{n = 0}^{N(I, \tau)-1}u_{\tau}^{(n)}1_{[t^{(n)}, t^{(n+1)})}(t), \\
        u_{\tau}(t) &:= \sum_{n = 0}^{N(I, \tau)-1}\left( \Exp_{u_{\tau}^{(n)}}( \ell_{\tau}^{(n)}(t) X_{\tau}^{(n)}) \right) 1_{[t^{(n)}, t^{(n+1)})}(t)
    \end{align*}
    for $t\in I$.
	In particular, $u_{\tau}$ is called the \textit{Rothe interpolation} of $\{ u_{\tau}^{(n)} \}_{n=0}^{N(I, \tau)}$.
\end{definition}
\begin{remark} 
	By definition, the Rothe interpolation $u_{\tau}$ is also represented as
	\begin{equation} \label{eq:3.3}
		u_{\tau} = \Exp_{\uu_{\tau}}(\ell_{\tau}\uX_{\tau}),
	\end{equation}
	where $\displaystyle \uX_{\tau} := \sum_{n = 0}^{N(I, \tau)-1}X_{\tau}^{(n)} 1_{[t^{(n)}, t^{(n+1)})}$. 
	Especially, $u_\tau$ is continuous in $I$.  
\end{remark}

\begin{proposition}[Properties of Rothe interpolation] \label{propertyofrothe} 
   Let $M$ be a path-connected and $C^2$-compact submanifold in $\RR^{\ell}$.
	Let $I := [0,T)$, $u_0 \in M_{\Delta}$ be an initial datum, $\tau >0$ be a step size, and  $\{ u_{\tau}^{(n)}\}_{n=0}^{N(I, \tau)}$ be the  sequence generated by the modified minimizing movement scheme $(\mathrm{MM}_{loc}; \tau, u_0)$. 
    Then, the Rothe interpolation $u_{\tau} : I \to M_{\Delta}$ of $\{ u_{\tau}^{(n)} \}_{n=0}^{N(I, \tau)}$ has the following properties:  
    \begin{enumerate}
        \item the curve $u_{\tau}$ is $C^{2}$-smooth in $I \setminus I(\tau)$, where $I(\tau)$ is defined in \eqref{eq:time_nodal_point},
        \item the velocity of $u_{\tau}$ is bounded by the  Lipschitz constant of $\btv_{\Delta}$, that is, 
        \begin{equation*}
            \left\| \frac{\id u_{\tau}}{\id t} \right\|_{H_{\Delta}} \leq \mathrm{Lip}(\btv_{\Delta}),
        \end{equation*}
        \item the acceleration of $u_{\tau}$ is bounded by the speed of $u_{\tau}$, that is,  
        \begin{equation*}
            \left\| \frac{\id^2 u_{\tau}}{\id t^2} \right\|_{H_{\Delta}} 
            \leq \tau^{-2} \cdot \mathrm{Curv}(M)\cdot \| \uX_{\tau} \|^2_{H_{\Delta}}.
        \end{equation*}
        Moreover, 
        \begin{equation*}
            \left\| \frac{\id^2 u_{\tau}}{\id t^2} \right\|_{H_{\Delta}} 
            \leq \mathrm{Curv}(M)\cdot \mathrm{Lip}(\btv_{\Delta})^2.
        \end{equation*}
    \end{enumerate}
\end{proposition}
\begin{proof}
    (1): Since \eqref{eq:3.3} implies  $u_{\tau}|_{[t^{(n-1)}, t^{(n)})} = \exp_{u^{(n-1)}_{\tau}}(\ell^{(n-1)}_{\tau}X_{\tau}^{(n-1)})$, $u_{\tau}$ is $C^2$-smooth in $(t^{(n-1)}, t^{(n)})$.
    
    \noindent
    (2): Since the  speed of the geodesic is constant, we have
    \begin{equation*}
            \left\| \frac{\id u_{\tau}}{\id t}\right\|_{H_{\Delta}} = \tau^{-1}\left\| X^{(n-1)}_{\tau} \right\|_{H_{\Delta}} \leq \mathrm{Lip}(\btv_{\Delta})
    \end{equation*}
    in $[t^{(n-1)},t^{(n)})$.
    Here, in the last inequality, we used Proposition \ref{proposition:scheme}.
    
    \noindent
    (3): Since $u_{\tau}^{(n-1)}$ and $u_{\tau}^{(n)}$ are joined by the exponential map, $u_{\tau}|_{[t^{(n-1)}, t^{(n)})}$ satisfies the geodesic equation \eqref{eq:sff}, that is, 
    \begin{align*}
        &\frac{\id^2 u_{\tau}}{\id t^2}(t) - \mathrm{II}_{u_{\tau}(t)}\left( \frac{\id u_{\tau}}{\id t}(t),\frac{\id u_{\tau}}{\id t}(t) \right) = 0, \\
        &u_{\tau}(t^{(n-1)}) = u_{\tau}^{(n-1)},
        \quad \frac{\id u_{\tau}}{\id t}(t^{(n-1)}) = \tau^{-1} X_{\tau}^{(n-1)}.
\end{align*}
    Hence, we have 
    \begin{equation*}
        \left\| \frac{\id^2 u_{\tau}}{\id t^2}(t) \right\|_{H_{\Delta}} = \left\| \mathrm{II}_{u_{\tau}(t)}\left( \frac{\id u_{\tau}}{\id t}(t),\frac{\id u_{\tau}}{\id t}(t) \right) \right\|_{H_{\Delta}} 
        \leq \mathrm{Curv}(M) \ \left\| \frac{\id u_{\tau}}{\id t}(t) \right\|^2_{H_{\Delta}}.
    \end{equation*}
    Since the geodesic has constant speed,
    \begin{equation*}
        \left\| \frac{\id^2 u_{\tau}}{\id t^2}(t) \right\|_{H_{\Delta}} 
        \leq \mathrm{Curv}(M) \cdot \left\| \frac{\id u_{\tau}}{\id t}(t) \right\|^2_{H_{\Delta}} 
        = \tau^{-2} \ \mathrm{Curv}(M) \ \left\| X_{\tau}^{(n-1)} \right\|^2_{H_{\Delta}}.
    \end{equation*}
    Moreover, Proposition \ref{proposition:scheme} implies that 
    \begin{equation*}
        \left\| \frac{\id^2 u_{\tau}}{\id t^2}(t) \right\|_{H_{\Delta}} \leq  \mathrm{Curv}(M)\  \mathrm{Lip}(\btv_{\Delta})^2.
        \smartqed
    \end{equation*}
\end{proof}

\subsubsection{Discrete energy dissipation property} \label{sec:energy_decay} 
The total variation dissipates along corresponding constrained TV flows.
Hence, it is desirable that the proposed scheme also has this property.
Indeed, the proposed scheme has this property if the step size is small enough.

\begin{proposition}[Energy dissipation]\label{thm:maindecay} 
	Let $M$ be a $C^2$-compact manifold embedded into $\RR^{\ell}$, $I := [0,T)$, $u_0 \in M_{\Delta}$ be an initial datum, $\tau >0$ be a step size and  $\{ u_{\tau}^{(n)}\}_{n=0}^{N(I, \tau)}$ be a sequence generated by the modified minimizing movement scheme $(\mathrm{MM}_{loc}; \tau, u_0)$.
    If $\tau \cdot \mathrm{Curv}(M) \cdot \mathrm{Lip}(\btv_{\Delta}) \leq 1$, then $\btv_{\Delta}(u_{\tau}^{(n+1)}) \leq \btv_{\Delta}(u_{\tau}^{(n)})$ holds for all $n = 0, \ldots, N(I, \tau)-1$.
\end{proposition}

\begin{proof} 
    Fix $n \in \{0, \ldots, N(I,\tau)-1 \}$.
    Then expanding $u^{(n+1)}_{\tau} = \Exp_{u^{(n)}_{\tau}}(X^{(n)}_{\tau})$ in Taylor series implies that
    \begin{equation*}
    u^{(n+1)}_{\tau} = u^{(n)}_{\tau} + X^{(n)}_{\tau} + \int_{t^{(n)}}^{t^{(n+1)}}(t^{(n+1)}-s) \frac{\id^2 u_{\tau}}{\id t^2}(s) \id s.
    \end{equation*}
    The above formula and the triangle inequality imply
    \begin{equation*}
        \btv_{\Delta}(u^{(n+1)}_{\tau}) \leq \btv_{\Delta}(u^{(n)}_{\tau} + X^{(n)}_{\tau}) + \btv_{\Delta}\left( \int_{t^{(n)}}^{t^{(n+1)}}(t^{(n+1)}-s) \frac{\id^2 u_{\tau}}{\id t^2}(s) \id s \right).
    \end{equation*}
    Since $X^{(n)}_{\tau}$ is the minimizer of $(\mathrm{VP}_{loc}; u^{(n)}_{\tau})$, we have
    \begin{equation*}
    \begin{split}
        \btv_{\Delta}(u^{(n+1)}_{\tau}) 
        \leq \btv_{\Delta}&(u^{(n)}_{\tau}) 
        - \frac{1}{2\tau}\|X^{(n)}_{\tau} \|^2_{H_{\Delta}}
        + \btv_{\Delta}\left( \int_{t^{(n)}}^{t^{(n+1)}}(t^{(n+1)}-s) \frac{\id^2 u_{\tau}}{\id t^2}(s) \id s \right).
        \end{split}
    \end{equation*}
    By applying Lipschitz continuity of $\btv_{\Delta}$ and the Minkowski  inequality for integrals $\left\|\int f\right\|_{H_\Delta}\le\int\|f\|_{H_\Delta}$, we have
    \begin{equation*}
    \begin{split}
        \btv_{\Delta}(u^{(n+1)}_{\tau}) 
        \leq &\btv_{\Delta}(u^{(n)}_{\tau}) 
        - \frac{1}{2\tau}\|X^{(n)}_{\tau} \|^2_{H_{\Delta}}\\
        &\hspace{50pt}+ \mathrm{Lip}(\btv_{\Delta}) \ \int_{t^{(n)}}^{t^{(n+1)}} (t^{(n+1)}-s) \left\| \frac{\id^2 u_{\tau}}{\id t^2}(s)\right\|_{H_{\Delta}}ds.
    \end{split}
    \end{equation*}
    Proposition \ref{propertyofrothe} implies that 
   \begin{equation*}
       \btv_{\Delta}(u^{(n+1)}_{\tau}) 
       \leq \btv_{\Delta}(u^{(n)}_{\tau}) + \frac12
        \mathrm{Lip}(\btv_{\Delta}) \ \mathrm{Curv}(M) \
       \|X^{(n)}_{\tau} \|^2_{H_{\Delta}} - \frac{1}{2\tau}\|X^{(n)}_{\tau} \|^2_{H_{\Delta}}.
   \end{equation*}
  Since $\tau \cdot \mathrm{Curv}(M) \cdot \mathrm{Lip}(\btv_{\Delta}) \leq 1$, we have $\btv_{\Delta}(u_{\tau}^{(n+1)})  \leq  \btv_{\Delta}(u_{\tau}^{(n)})$.
\end{proof}

\subsubsection{Error estimate}\label{sec:convergence} 

Here, we establish an error estimate between the sequence generated by $(\mathrm{MM}_{loc}; \tau, u_0)$ and the solution to $(\mathrm{DTVF}_{\mathrm{GK}};u_0)$ when $u_0 ˆ\in M_{\Delta}$ is given.

Now  we state the error estimate of the numerical solution $\{ u^{(n)}_{\tau} \}_{n=0}^{N(I, \tau)} \subset M_{\Delta}$.

\begin{theorem}[Error estimate]\label{thm:mainerror} 
	Let $M$ be a path-connected and $C^2$-compact submanifold in $\RR^{\ell}$, $I := [0,T)$ and $\tau>0$.
    Fix two initial data $u_0^1, u_0^2 \in M_{\Delta}$.
    Let $u \in C(I;M_{\Delta})$ be a solution of the discrete GK model $(\mathrm{DTVF}_{\mathrm{GK}}; u_0^1)$ and $u_{\tau} \in C(I;M_{\Delta})$ be the Rothe interpolation of the series $\{ u_{\tau}^{(n)} \}_{n = 0}^{N(I, \tau)}$ given by $\left(MM_{loc};\tau,u^2_0\right)$. 
     Then, 
    \begin{equation}\label{eq:error}
        \norm{u_{\tau}(t) - u(t)}^2_{H_{\Delta}} \leq e^{C_0t}\norm{u^1_0 - u^2_0}_{H_{\Delta}}^2 + te^{C_0t}(C_1\tau+C_2\tau^2)
    \end{equation}
    for all $t \in I$, where
    \begin{align}
        C_0 &:= 2C_M\Lip(\btv_{\Delta})v(\Omega_{\Delta})^{-1/2}, \label{eq:C1}\\
        C_1 &:= \left(2+\mathrm{Diam}(M) \mathcal{L}^k (\Omega)^{1/2} \left(2C_Mv (\Omega_\Delta)^{-1} +\mathrm{Curv}(M) \right)\right) \Lip(\boldsymbol{TV}_\Delta)^2, \label{eq:C2}\\
        C_2 &:= \left( \frac32\mathrm{Curv}(M) + C_M v(\Omega_{\Delta})^{-1/2}\right) \ \mathrm{Lip}(\btv_{\Delta})^3. \label{eq:C3}
    \end{align}
    
    Here, we again note that the constant $C_M$ which appears in \eqref{eq:Curv+LSF} can be bounded by quantities only associated with the manifold $M$; that is, all the constants $C_0$, $C_1$ and $C_2$ are independent of time $t$ and step size $\tau$.
\end{theorem}
\begin{remark} 
    In this theorem, we cannot remove the exponentially growing  term $e^{C_0t}$ from the right hand side of~\eqref{eq:error}, because the functional $\btv_\Delta$ defined by~\eqref{eq:discretized-TV} is not in general convex but merely semi-convex. 
\end{remark}

We immediately see that if $u^1_0 = u^2_0$ in \eqref{eq:error}, then we have
\begin{equation*}
    \norm{u_{\tau}(t) - u(t)}^2_{H_{\Delta}} \leq te^{C_0t}(C_1\tau + C_2\tau^2),
\end{equation*}
and thus we conclude  that $u_\tau$ converges to $u$ in $C(I;M_{\Delta})$ as $\tau \to 0$.

The key estimates to prove Theorem~\ref{thm:mainerror} are the \textit{evolutionary  variational inequalities} for $u_{\tau}$ and $u$.
\begin{proposition}\label{prop:main1} 
    For each $v \in M_{\Delta}$,
    the Rothe interpolation $u_{\tau}: I \to M_{\Delta}$ of $\{ u_{\tau}^{(n)} \}_{n=0}^{N(I, \tau)}$ satisfies 
    \begin{equation}\label{eq:EVI-Rothe}
        \frac{1}{2}\frac{\id}{\id t} \| u_{\tau} - v \|^2_{H_{\Delta}} \leq \btv_{\Delta}(v)  - \btv_{\Delta}(u_{\tau})
        + \frac{C_0}{2} \| u_{\tau} - v \|^2_{H_{\Delta}} +  C_1\tau + C_2\tau^2
    \end{equation}
    for all $t \in I \setminus I(\tau)$, where $C_0, C_1, C_2$ are the same constants as in Theorem \ref{thm:mainerror}, and $I(\tau)$ is defined in \eqref{eq:time_nodal_point}.
\end{proposition}
\begin{proof}

    We will compute $\displaystyle \frac{1}{2}\frac{\id}{\id t} \norm{u_{\tau} - v}^2_{H_{\Delta}}$
    by splitting it into a semi-monotone term and an error term.
    Expanding $\id u_{\tau}/\id t$ in Taylor series implies that 
    \begin{equation*}
        \frac{\id u_{\tau}}{\id t} = \frac{1}{\tau}\uX_{\tau} + \int_{t^{(n)}}^t \frac{\id^2 u_{\tau}}{\id t^2}(s)\id s,
    \end{equation*}
    and inserting this into
    \begin{equation*}
        \frac{1}{2}\frac{\id}{\id t} \| u_{\tau} - v \|^2_{H_{\Delta}}
        = \inner{\frac{\id u_{\tau}}{\id t}}{u_{\tau} - v}{{H}_{\Delta}},
    \end{equation*}
    we have 
    \begin{equation}\label{eq:splitI&II}
        \begin{split}
            \frac{1}{2}\frac{\id}{\id t} \| u_{\tau} - v \|^2_{H_{\Delta}}
            &= \inner{\frac{1}{\tau}\uX_{\tau}}{u_{\tau} - v}{H_{\Delta}} + \int_{t^{(n)}}^t \inner{\frac{\id^2 u_{\tau}}{\id t^2}(s)}{u_{\tau} - v}{H_{\Delta}}\id s  \\
            &=: \bm{I} + \bm{II}. 
        \end{split}
    \end{equation}
    Moreover, we split the term $\bm{I}$ into a monotone term and a non-monotone term.
    Proposition \ref{proposition:scheme} implies that
    \begin{equation*}
        \bm{I} 
        = \inner{-P_{\uu_{\tau}} \partial \btv_{\Delta}(\uu_{\tau} + \uX_{\tau})}{u_{\tau} - v}{H_{\Delta}}.
    \end{equation*}
    Self-adjointness  of $P_{\uu_{\tau}}$ and the relation $P_{\uu_{\tau}} = I - P^{\bot}_{\uu_{\tau}}$ imply that
    \begin{equation}\label{eq:splitI_1&I_2}
        \begin{split}
            \bm{I} 
            &= \inner{\partial\btv_{\Delta}(\uu_{\tau} + \uX_{\tau})}{v - u_{\tau}}{H_{\Delta}} 
            + \inner{\partial \btv_{\Delta}(\uu_{\tau} + \uX_{\tau})}{ P^{\bot}_{\uu_{\tau}}(u_{\tau}-v) }{H_{\Delta}} \\
            &=: \bm{I}_1 + \bm{I}_2. 
        \end{split}
    \end{equation}\label{eq:I_1I_2II}
    We plug \eqref{eq:splitI_1&I_2} into the identity \eqref{eq:splitI&II}  to obtain
    \begin{equation}\label{eq:split}
        \frac{1}{2}\frac{\id}{\id t} \| u_{\tau} - v \|^2_{H_{\Delta}}
        = \bm{I}_1 + \bm{I}_2 + \bm{II}.
    \end{equation}
    \\
    We shall estimate $\bm{I}_1$, $\bm{I}_2$ and $\bm{II}$, respectively.
    First, we estimate the term $\bm{I}_1$:
    \begin{align*}
        \bm{I}_1 
        &= \left\langle \partial \btv_{\Delta}(\uu_{\tau} + \uX_{\tau}), (v - u_{\tau} + \uu_{\tau} + \uX_{\tau}) - (\uu_{\tau} + \uX_{\tau}) \right\rangle_{H_{\Delta}} \\
        &\leq \btv_{\Delta}(v - u_{\tau} + \uu_{\tau} + \uX_{\tau}) - \btv_{\Delta}(\uu_{\tau} + \uX_{\tau}).
    \end{align*}

    Since $u_{\tau} = \uu_{\tau} + \tau^{-1}(t - t^{(n)})\uX_{\tau} + \int_{t^{(n)}}^{t}(t-s)\frac{\id^2 u_{\tau}}{\id t^2}(s) \id s$, we have
    \begin{equation*}
    \begin{split}
        \bm{I}_1
        \leq \btv_{\Delta}&\left( v + \left(1 - \frac{t - t^{(n)}}{\tau}\right)\uX_{\tau} - \int_{t^{(n)}}^{t}(t-s)\frac{\id^2 u_{\tau}}{\id t^2}(s) \id s \right) \\
        &- \underline{\btv_{\Delta}\left( u_{\tau} + \left(1 - \frac{t - t^{(n)}}{\tau}\right)\uX_{\tau} -  \int_{t^{(n)}}^{t}(t-s)\frac{\id^2 u_{\tau}}{\id t^2}(s) \id s \right)}.
    \end{split}
    \end{equation*}
    Here, we apply  the triangle  inequality, $\btv_{\Delta}\left( x - y \right) \geq |\btv_{\Delta}\left( x \right) -\btv_{\Delta}\left( y \right) |$ for $x, y \in H_{\Delta}$,  
    to the underlined part in the above inequality to obtain
    \begin{equation*}
    \begin{split}
        \bm{I}_1 
        \leq  \btv_{\Delta}&\left( v + \left(1 - \frac{t - t^{(n)}}{\tau}\right)\uX_{\tau} - \int_{t^{(n)}}^{t}(t-s)\frac{\id^2 u_{\tau}}{\id t^2}(s) \id s \right) \\
        &-  \btv_{\Delta}( u_{\tau}) +  \btv_{\Delta}\left(-\left(1 - \frac{t - t^{(n)}}{\tau}\right)\uX_{\tau} +  \int_{t^{(n)}}^{t}(t-s)\frac{\id^2 u_{\tau}}{\id t^2}(s) \id s \right).
    \end{split}
    \end{equation*}
    The triangle inequality for $\btv_{\Delta}$ implies that
    \begin{equation*}
    \begin{split}
        \bm{I}_1
        \leq \btv_{\Delta}&(v) - \btv_{\Delta}(u_{\tau}) \\ &+ 2 \ \btv_{\Delta}\left(\left(1 - \frac{t - t^{(n)}}{\tau}\right)\uX_{\tau} \right) 
        + 2 \ \btv_{\Delta}\left( \int_{t^{(n)}}^{t}(t-s)\frac{\id^2 u_{\tau}}{\id t^2}(s) \id s \right).
    \end{split}
    \end{equation*}
    Proposition \ref{property:dtv} implies 
    \begin{equation*}
    \begin{split}
        \bm{I}_1
        &\leq \btv_{\Delta}(v) - \btv_{\Delta}(u_{\tau}) \\
        &\hspace{-10pt}+2\mathrm{Lip}(\btv_{\Delta}) \left\| \left(1 - \frac{t - t^{(n)}}{\tau}\right)\uX_{\tau} \right\|_{H_{\Delta}} 
       +2 \ \mathrm{Lip}(\btv_{\Delta})  \int_{t^{(n)}}^{t}(t-s)\left\|\frac{\id^2 u_{\tau}}{\id t^2}(s)\right\|_{H_{\Delta}} \id s.
    \end{split}
    \end{equation*}
    Proposition \ref{propertyofrothe} implies 
    \begin{equation}\label{ine:I_1} 
        \bm{I}_1
        \leq \btv_{\Delta}(v) - \btv_{\Delta}(u_{\tau}) + 2 \ \mathrm{Lip}(\btv_{\Delta})^2\tau + \mathrm{Curv}(M) \ \mathrm{Lip}(\btv_{\Delta})^3\tau^2.
    \end{equation}
    Next, we estimate the term $\bm{I}_2$.
    The Cauchy--Schwarz inequality and Proposition \ref{property:dtv} imply that
    \begin{equation}\label{ineq:II_0}
        \begin{split}
            \bm{I}_2 
            &\leq \mathrm{Lip}(\btv_{\Delta}) \cdot \norm{P^{\bot}_{\uu_{\tau}}(v - u_{\tau})}_{H_{\Delta}}. 
        \end{split}
    \end{equation}
    Here, we claim that
    \begin{equation}\label{ineq:claim01}
        \norm{P^{\bot}_{\uu_{\tau}}(v - u_{\tau})}_{H_{\Delta}} 
        \leq D_0\|  v - u_{\tau} \|_{H_{\Delta}}^2 + D_1\tau \left\|  v - u_{\tau} \right\|_{H_{\Delta}} + D_2\tau^2,
    \end{equation}
    where
    \begin{align*}
        D_0 &= C_M \ v(\Omega_{\Delta})^{-1/2}, \\
        D_1 &= 2 \ C_M \ v(\Omega_{\Delta})^{-1/2} \ \mathrm{Lip}(\btv_{\Delta}),  \\
        D_2 &= \left( \frac{\mathrm{Curv(M)}}{2} + C_M \ v(\Omega_{\Delta})^{-1} \right) \ \mathrm{Lip}(\btv_{\Delta})^2.
    \end{align*}
    Indeed, since  $u_{\tau} = \uu_{\tau} + \tau^{-1}(t - t^{(n)})\uX_{\tau} + \int_{t^{(n)}}^{t}(t-s)\frac{\id^2 u_{\tau}}{\id t^2}(s) \id s$, we have
    \begin{align*}
        P^{\bot}_{\uu_{\tau}}(v - u_{\tau}) 
        &= P^{\bot}_{\uu_{\tau}}( v - \uu_{\tau}) - P^{\bot}_{\uu_{\tau}}\left( \int_{t^{(n)}}^{t}(t-s)\frac{\id^2 u_{\tau}}{\id t^2}(s) \id s \right).
    \end{align*}
    Taking the norm $\| \cdot \|_{H_{\Delta}}$ in the above equation and applying the triangle inequality and the Minkowski  inequality for integrals yield  
    \begin{equation}\label{ineq:3}
        \| P^{\bot}_{\uu_{\tau}}(v - u_{\tau}) \|_{H_{\Delta}} 
        \leq \| P^{\bot}_{\uu_{\tau}}( v - \uu_{\tau}) \|_{H_{\Delta}}  + \int_{t^{(n)}}^{t}(t-s) \left\|  \frac{\id^2 u_{\tau}}{\id t^2}(s) \right\|_{H_{\Delta}} \id s.
    \end{equation}
    As the second term on the right-hand side is estimated in the same way as in \eqref{ine:I_1},  we focus on the term $\| P^{\bot}_{\uu_{\tau}}( v - \uu_{\tau}) \|_{H_{\Delta}}$.
    Since $C_M$ in \eqref{eq:Curv+LSF} is bounded, we see that 
\[
	\left| \pi^{\bot}_{\uu_c} (v-{\uu_c}) \right|^2
	\leq C^2_M \left| v-{\uu_\tau} \right|^4
\]
pointwise.
 Thus \eqref{eq:Curv+LSF} yields 
    \begin{equation}\label{ineq:3-1}
        \| P^{\bot}_{\uu_{\tau}}( v - \uu_{\tau}) \|_{H_{\Delta}}
        \leq C_M \| v - \uu_{\tau} \|_{L^4}^2.
    \end{equation}
Moreover, for sequences it is clear that
\[
	\sum|a_i|^4 \leq \left(\sum|a_i|^2 \right)^2.
\]
Thus,
\[
	\|f\|_{L^4} \leq \nu (\Omega_\Delta)^{-1/4 }\|f\|_{H_\Delta}
\]
for $f \in H_\Delta$, and we obtain that 
   \begin{equation}
        \| P^{\bot}_{\uu_{\tau}}( v - \uu_{\tau}) \|_{H_{\Delta}}
        \leq C_M v(\Omega_{\Delta})^{-1/2} \| v - \uu_{\tau} \|_{H_{\Delta}}^2,
    \end{equation}
    where $v(\Omega_{\Delta})$ is defined in \eqref{eq:inf_sup_volume}.
    We split $v - \uu_{\tau} = v - u_{\tau} + u_{\tau} - \uu_{\tau}$ to obtain
    \begin{equation*}
        \| P^{\bot}_{\uu_{\tau}}( v - \uu_{\tau}) \|_{H_{\Delta}} 
        \leq  C_M v(\Omega_{\Delta})^{-1/2} (\|  v - u_{\tau} \|_{H_{\Delta}}^2 + 2\left\langle  v - u_{\tau}, u_{\tau} - \uu_{\tau} \right\rangle_{H_{\Delta}} + \|  u_{\tau} - \uu_{\tau} \|_{H_{\Delta}}^2).
    \end{equation*}
    The Cauchy--Schwarz inequality implies that 
    \begin{equation*}
    \begin{split}
        \| P^{\bot}_{\uu_{\tau}}( v - \uu_{\tau}) \|_{H_{\Delta}}
        \leq  C_M v(\Omega_{\Delta})^{-1/2} (&\|  v - u_{\tau} \|_{H_{\Delta}}^2\\ 
        &+ 2\left\|  v - u_{\tau} \right\|_{H_{\Delta}}\left\| u_{\tau} - \uu_{\tau} \right\|_{H_{\Delta}} + \|  u_{\tau} - \uu_{\tau} \|_{H_{\Delta}}^2).
    \end{split}
    \end{equation*}
    Proposition \ref{propertyofrothe} implies 
    \begin{equation}
        \begin{split}
        \| P^{\bot}_{\uu_{\tau}}( v - \uu_{\tau}) \|_{H_{\Delta}} 
        \leq  C_M &v(\Omega_{\Delta})^{-1/2} (\|  v - u_{\tau} \|_{H_{\Delta}}^2\\ 
        &+ 2\tau \ \mathrm{Lip}(\btv_{\Delta}) \ \left\|  v - u_{\tau} \right\|_{H_{\Delta}} + \tau^2 \ \mathrm{Lip}(\btv_{\Delta})^2).
        \end{split}
        \label{eq:ineq}
    \end{equation}

    Plugging \eqref{eq:ineq} into \eqref{ineq:3} yields the claim \eqref{ineq:claim01}, and taking into account \eqref{ineq:II_0}, we obtain
    
    \begin{equation}\label{ine:I_2}
        \bm{I}_2 
        \leq \mathrm{Lip}(\btv_{\Delta}) (D_0\|  v - u_{\tau} \|_{H_{\Delta}}^2 + D_1\tau \left\|  v - u_{\tau} \right\|_{H_{\Delta}} + D_2\tau^2).
    \end{equation}
    Next, we estimate the term $\bm{II}$.
    The Cauchy--Schwarz inequality and the Minkowski  inequality for integrals imply 
    \begin{equation*}
        \bm{II} \leq \int_{t^{(n)}}^t \norm{\frac{\id^2 u_{\tau}}{\id t^2}(s)}_{H_{\Delta}}\id s \norm{u_{\tau} - v}_{H_{\Delta}}.
    \end{equation*}
    Proposition \ref{propertyofrothe} yields 
    \begin{equation}\label{ine:II}
        \bm{II} \leq \tau \ \mathrm{Curv}(M)\ \mathrm{Lip}(\btv_{\Delta})^2 \ \norm{u_{\tau} - v}_{H_{\Delta}}.
    \end{equation}
The term $\tau\|u_\tau-v\|_{H_\Delta}$ is estimated by $\tau\,\Diam(M)\mathcal{L}^k(\Omega)^{1/2}$ since $M$ is compact.

    Finally, we combine inequalities \eqref{eq:split}, \eqref{ine:I_1}, \eqref{ine:I_2} and \eqref{ine:II} to obtain
    \begin{align*}
        \frac{1}{2}\frac{\id}{\id t} \| u_{\tau} - v \|^2_{H_{\Delta}} 
        &\leq \btv_{\Delta}(v ) - \btv_{\Delta}(u_{\tau}) + \frac{C_0}2\norm{u_{\tau} - v}^2_{H_{\Delta}} +  C_1\tau + C_2\tau^2,
    \end{align*}
     where $C_0$, $C_1$ and $C_2$ are given in \eqref{eq:C1}, \eqref{eq:C2} and \eqref{eq:C3}, respectively, which completes the proof.
\end{proof}

The solution $u$ of $(\mathrm{DTVF}_{GK}; u_0)$ satisfies the following evolutionary  variational inequality which is obtained by an  argument similar to that in the proof of Proposition \ref{prop:main1}; for the proof, see \cite{Taguchi2018}. 
\begin{proposition}\label{prop:main2} 
    For each $v \in M_{\Delta}$, the solution $u \in W^{1,2}(I;M_{\Delta})$ of discrete GK model  $(\mathrm{DTVF}_{\mathrm{GK}};u_0)$ satisfies 
    \begin{equation}\label{eq:EVI-DTVF}
        \frac{1}{2}\frac{\id}{\id t}\norm{u - v}^2_{H_{\Delta}} \leq \btv_{\Delta}(v) - \btv_{\Delta}(u) + \frac{C_0}{2}\norm{u - v}^2_{H_{\Delta}}
    \end{equation}
    for a.e. $t \in (0,T)$, where $C_0$ is the same constant as  in Proposition~\ref{prop:main1}.
\end{proposition}

Now, we will finish the proof of Theorem~\ref{thm:mainerror}.

\noindent
\textbf{Proof of Theorem~\ref{thm:mainerror}:}
    Fix $t \in (0,T)$.
    By substituting $v = u_{\tau}(t)$ into \eqref{eq:EVI-DTVF} 
    and $v = u(t)$ into \eqref{eq:EVI-Rothe}
    and adding these two inequalities,
    we obtain
    \begin{equation*}
        \frac{\id}{\id t} \norm{u(t) - u_{\tau}(t)}^2_{H_{\Delta}} 
        \leq C_0\norm{u(t) - u_{\tau}(t)}^2_{H_{\Delta}} + C_1 \tau + C_2\tau^2,
    \end{equation*}
    where $C_0$, $C_1$ and $C_2$ are given in \eqref{eq:C1}, \eqref{eq:C2} and \eqref{eq:C3}, respectively.
    By the Gronwall's inequality, we have
    \begin{equation*}
        \norm{ u_{\tau}(t) - u(t) }_{H_{\Delta}}^2 \leq e^{C_0t}(\norm{ u_{\tau}(0) - u(0) }^2_{H_{\Delta}} + t(C_1\tau + C_2\tau^2))
    \end{equation*}
    for all $t \in [0,T)$, which proves Theorem~\ref{thm:mainerror}.
\qed

\begin{remark}[Convergence of spatially discrete TV flow as the mesh tends to zero]
This problem is difficult and is studied only for unconstrained problems.
We consider the case for the TV flow of scalar functions.
 If one uses a rectangular partition of $\Omega$, then the solution of discrete model solves an anisotropic $\ell_1$-total variation flow, which is the gradient flow of
\[
	TV_{\ell_1}(u) := \int_\Omega |\nabla u|_{\ell_1} dx,
\]
where $|p|_{\ell_1}:=\sum^k_{i=1}|p_i|$, $p=(p_1, \ldots, p_k)$.
 This is known for the one-dimensional case for a long time ago \cite{GGK} and for higher dimensional case by \cite{LMM}.
 Thus, a discrete solution $u_h$ converges to a solution $u$ of
\[
	\frac{\partial u}{\partial t} = \sum^k_{j=1} \frac{\partial}{\partial x_j}
	\frac{\partial u/\partial x_j}{|\partial u/\partial x_j|}
\]
in $L^\infty \left(0,T,L^2(\Omega)\right)$ if initially $\left. u_h \right|_{t=0}=u_{0h}$ converges to the continuum initial data $u_0$ in $L^2(\Omega)$.
 More precisely,
\[
	\left\| u_h - u \right\|_{L^2}(t) \leq \left\| u_{0h} - u_0 \right\|_{L^2}
\]
holds since the solution semigroup is a contraction semigroup.
For constrained cases, it is expected, but so far, there is no literature stating this fact.
\end{remark}

\section{Practical algorithms}\label{sec:numerical} 

In this section, we present the numerical algorithms obtained  by the proposed scheme.

\subsection{The proposed scheme with alternating split Bregman iteration} 
In order to implement  the proposed scheme, we need to solve a minimization problem in each iteration.
We simplify this optimization problem by applying  alternating split Bregman iterations.
We replace the minimization problem $(\mathrm{VP_{loc}}; u_{\tau}^{(n-1)})$ with alternating split Bregman iteration which is proposed in~\cite{Goldstein2009} to solve the $L^1$ regularization problem efficiently.

First we apply a splitting method to $(\mathrm{VP}_{loc}; u_{\tau}^{(n-1)})$, and we obtain the split formulation $(\mathrm{VP}_{loc, split}; u_{\tau}^{(n-1)})$:
\begin{equation*}
    \min_{X \in H_0, Z \in H_1} \Phi_{loc, split}^{\tau}(X,Z; Y^{(n-1)}) \ \mbox{subject to} \ Z = \mathcal{W}(X) + Y^{(n-1)},
\end{equation*}
where $\Phi_{loc, split}^{\tau}(\cdot, \cdot ; Y^{(n-1)}): H_0 \times H_1 \to \RR\cup\{ \infty \}$ is defined as
\begin{equation*}
    \begin{aligned}
        \Phi_{loc, split}^{\tau}(X, Z; Y^{(n-1)}) &:= \tau \sum_{\gamma \in e(\Delta)} \|Z_0^{\gamma} \|_{\RR^{\ell}} \mathcal{H}^{k-1}(E_{\gamma})
        + \tau I_{T_{u_{\tau}^{(n-1)}}M_{\Delta}}(Z_1) + \frac{1}{2}\norm{X}_{H_{\Delta}}^2,\\
        I_{T_{u_{\tau}^{(n-1)}}M_{\Delta}}(X) &:= %
        \left\{
            \begin{array}{cc}
                0, &\mbox{if $X \in T_{u_{\tau}^{(n-1)}}M_{\Delta}$,} \\
                \infty, & \mbox{otherwise},
            \end{array}
        \right.
        \quad
        \mathcal{W}(X) := \pmt{\bfD_{\Delta} X \\ X}, \\
        Y^{(n-1)} &:= \pmt{\bfD_{\Delta} u_{\tau}^{(n-1)} \\ 0},
        \quad
        Z = \pmt{Z_0 \\ Z_1},
    \end{aligned}
\end{equation*}
in which $H_0 := H_{\Delta}$ and $H_1 := H_{E\Omega_{\Delta}} \times H_{\Delta}$.
Subsequently, we apply Bregman iteration with alternating minimization method to $(\mathrm{VP}_{loc, split};u_{\tau}^{(n-1)})$, to arrive at the following algorithm: 

\begin{algorithm} 
[$(\mathrm{VP}_{loc, split};u_{\tau}^{(n-1)})$: Alternating Split Bregman Iteration for $(\mathrm{VP}_{loc};u_{\tau}^{(n-1)})$]
    Set $\displaystyle X_{\tau}^{(n-1)} := \lim_{k \to \infty} X^{(k)}$, where the sequence $\{ X^{(k)} \}_{k=0}^{\infty}$ is defined by the following procedure:
    \begin{enumerate}
        \item For $k=0$: Set $\rho>0$, $Z^{(0)} \in H_1$ and $B^{(0)} \in H_1$.
        \item For $k \geq 1$: 
            \begin{enumerate} 
                \item   $X^{(k)} := \argmin_{X \in H_0}  \Phi_{loc, SBI}^{\tau}\left(X, Z^{(k-1)}, B^{(k-1)}; Y^{(n-1)}\right)$,
                \item   $Z^{(k)} := \argmin_{Z \in H_1} \Phi_{loc, SBI}^{\tau}\left(X^{(k)} ,Z, B^{(k-1)}; Y^{(n-1)}\right)$,
                \item   $B^{(k)} := B^{(k-1)} +  \mathcal{W}(X^{(k)}) + Y^{(n-1)} -  Z^{(k)}$.
            \end{enumerate}
    \end{enumerate}
    Here, $\Phi_{loc, SBI}^{\tau}(\cdot, \cdot, \cdot ;Y^{(n-1)}) : H_0 \times H_1 \times H_1 \to \RR \cup \{ \infty \}$ is defined as 
    \begin{equation*}
        \Phi_{loc, SBI}^{\tau}\left(X, Z, B ;Y^{(n-1)} \right) := 
        \Phi_{loc, split}^{\tau}\left(X, Z; Y^{(n-1)} \right) + \frac{\rho}{2} \norm{Z - \mathcal{W}(X) - Y^{(n-1)} - B}_{H_1}^2,
    \end{equation*}
    
    where $\|Z\|_{H_1}^2=\|Z_0\|_{H_\Delta}^2+\|Z_1\|_{H_{E\Omega_\Delta}}^2$ and $\|Z_1\|_{H_{E\Omega_\Delta}}^2$ is defined by
    \begin{align*}
    	\|Z_1\|_{H_{E\Omega_\Delta}}^2
    	=
    	\sum_{\gamma\in e(\Delta)}\|Z_1^\gamma\|_{\mathbb{R}^\ell}^2\mathcal{H}^{k-1}(E_\gamma).
    \end{align*}
\end{algorithm}

Here, we note that
\begin{enumerate}
	\renewcommand{\labelenumi}{(\roman{enumi})}
    \item both $X^{(k)}$ and $Z^{(k)}$ in the above iterations are solved explicitly when the orthogonal projection and the exponential map in $M$ have explicit formulae.
    This is the case, for example, for the class of orthogonal Stiefel manifolds, which includes the important manifolds $S^2$ and $SO(3)$.  
    \item $X^{(k)}$ converges to the minimizer of $(\mathrm{VP}_{loc};u_{\tau}^{(n-1)})$ in $H_0$ thanks to Corollary 2.4.10 in \cite{Setzer2009}.
\end{enumerate}

We explain details concerning the first point (i).
In step (a), we seek the global minimizer $X^{(k)}$ of the function
\begin{align*}
	\frac{1}{2}\|X\|_{L^2(\Omega_\Delta)}^2+\frac{\rho}{2}\left\|Z-\mathcal{W}(X)-Y^{(n-1)}-B\right\|_{H_1}^2,
	\quad
	X\in H_\Delta.
\end{align*}
Differentiating this function with respect to $X$, we see that the global minimizer can be characterized as the solution to the corresponding Poisson equation, which is a strictly diagonally dominant system and can be efficiently solved by using well-known solvers such as the Gauss--Seidel method.
In step (b), we first seek $Z_0^{(k)}$ which is the global minimizer in $H_{E\Omega_\Delta}$ of the function
\begin{align*}
	&\tau\|Z_0\|_{H_{E\Omega_\Delta}}+\frac{\rho}{2}\|Z_0-\bm{D}_\Delta X^{(k)}-\bm{D}_\Delta u_\tau^{(n-1)}-B_0^{(k-1)}\|_{H_{E\Omega_\Delta}}\\
	&=\sum_{\gamma\in e(\Delta)}\left(\tau\|Z_0^\gamma\|_{\mathbb{R}^\ell}+\frac\rho2\|Z_0^\gamma-(\bm{D}_\Delta X^{(k)})^\gamma-(\bm{D}_\Delta u_\tau^{(n-1)})^\gamma-(B_0^{(k-1)})^\gamma\|_{\mathbb{R}^\ell}^2\right)\mathcal{H}^{k-1}(E_\gamma).
\end{align*}
Since there are no interactions between $\{Z_0^\gamma\}_{\gamma\in e(\Delta)}$ in the above equation, the global minimizer can be obtained by computing it componentwise.
Thus, we need to obtain the global minimizer $x^*\in\mathbb{R}^\ell$ of the function $\tau\|x\|_{\mathbb{R}^\ell}+(\rho/2)\|x-y\|_{\mathbb{R}^l}^2$ ($y\in\mathbb{R}^\ell$), and we can explicitly write down the global minimizer $x^*\in\mathbb{R}^\ell$ as
\begin{align*}
	x^*=\shrink\left(y,\frac\tau\rho\right),
\end{align*}
where $\shrink(x,\gamma)$ ($x\in\mathbb{R}^\ell$, $\gamma\in\mathbb{R}$) is the shrinkage operator defined by
\begin{align*}
	\shrink(x,\gamma)
	=
	\frac{x}{\|x\|_{\mathbb{R}^\ell}}\max(\|x\|_{\mathbb{R}^\ell}-\gamma,0).
\end{align*}
As a result, each $(Z_0^{(k)})^\gamma$ is given by
\begin{align*}
	(Z_0^{(k)})^\gamma
	=
	\shrink\left(A^\gamma,\frac\tau\rho\right),
	\quad
	A^\gamma
	:=
	(\bm{D}_\Delta X^{(k)})^\gamma+(\bm{D}_\Delta u_\tau^{(n-1)})^\gamma+(B_0^{(k-1)})^\gamma.
\end{align*}
We finally compute $Z_1^{(k)}$, which is the global minimizer of the function
\begin{align*}
	\tau I_{T_{u_\tau^{(n-1)}}M_\Delta}(Z_1)+\frac{\lambda}{2}\|Z_1-X^{(k)}-B_1^{(k-1)}\|_{H_\Delta}^2.
\end{align*}

The global minimizer for this equation can be given by using the orthogonal projection as follows:
\begin{align*}
	Z_1^{(k)}
	=
	P_{u_\tau^{(n-1)}}(X^{(k)}+B_1^{(k-1)}).
\end{align*}
Summarizing the above, in the alternating split Bregman iteration, we do not need to solve minimization problems directly, and each step requires only solving a well-conditioned linear system, performing an algebraic manipulation and an orthogonal projection.

Finally, we state the proposed scheme with alternating split Bregman iteration $(\mathrm{MM}_{loc, SBI}; \tau, u_0)$:

\begin{algorithm}[$(\mathrm{MM}_{loc, SBI}; \tau, u_0)$: Proposed Scheme with Alternating Split Bregman Iteration]
    Let $u_0 \in M_{\Delta}$ and $\tau>0$ be a time step size.
    Then,  we define the sequence $\{u_{\tau}^{(n)} \}_{n=0}^{N(I, \tau)}$ in $M_{\Delta}$ by the following procedure:
    \begin{enumerate}
        \item For $n=0$: $u_{\tau}^{(0)}  := u_{0}$.
        \item For $n \geq 1$: $u_{\tau}^{(n)}$ is defined by the following steps:
        \begin{enumerate}
            \item Set $\displaystyle X_{\tau}^{(n-1)} := \lim_{k \to \infty} X^{(k)} \in T_{u_{\tau}^{(n-1)}}M_{\Delta}$, where the sequence $\{ X^{(k)} \}_{k=0}^{\infty}$ is  obtained by algorithm $(\mathrm{VP}_{\mathrm{loc,split}};u_\tau^{(n-1)})$.
            	
            \item Set $u_{\tau}^{(n)}=\Exp_{u_{\tau}^{(n-1)}}(X_{\tau}^{(n-1)})$.
        \end{enumerate}
    \end{enumerate}
\end{algorithm}

\begin{remark}
	As we have explained in the above, the computational cost of our method based on alternating split Bregman iteration is cheap.
	Indeed, as pointed out in \cite{Goldstein2009}, if we choose parameters properly, the number of iterations in Algorithm $(\mathrm{VP}_{loc,split};u_\tau^{(n-1)})$ becomes small; however, as far as we know, there are no mathematical guidelines on optimal choice of parameters (see \cite{Goldstein2009,GU} for detailed explanation on the choice of the initial values $Z^{(0)}$ and $B^{(0)}$ and the parameter $\rho$).
	
	For constrained TV flows, several numerical methods have been proposed based on the regularization of the total variation energy or Lagrange multipliers.
	If we impose a regularization of the total variation energy, the problem on singularity disappears but it becomes impossible to capture the steep edge structure and the facet-preserving phenomenon.
	When the target manifold $M$ is the sphere $S^{l-1}$, then it is not so difficult to obtain the unconstrained problem by introducing the Ginzburg--Landau functionals or Lagrange multipliers.
	However, if we focus on more complicated manifolds such as $SO(3)$, introducing these kinds of regularizations is not straightforward; also, it is not clear whether or not we can construct a practical numerical scheme to simulate it.
	Our approach does not regularize but merely convexify the total variation energy, and we adopted the orthogonal projection and the exponential map to obtain the solution at next time step.
	Namely, our method does not restrict the target manifold a priori and can be applied to a broad class of target manifolds.
\end{remark}

\section{Numerical results}
\label{sec:numerical_results}

In this section, we use the above scheme to simulate $S^2$ and $SO(3)$ valued TV flows, respectively.
Throughout the numerical experiments, we choose initial values $Z^{(0)}$ and $B^{(0)}$ in algorithm $(\mathrm{VP}_{loc,split};u_\tau^{(n-1)})$ as zeros, and set the parameter $\rho$ to be equal to $0.1$ (see \cite{Goldstein2009} for detailed explanation on the choices of the initial values $Z^{(0)}$ and $B^{(0)}$ and the parameter $\rho$).
Moreover, we terminate the iteration for computing $X_\tau^{(n-1)}$ when the relative error becomes less than $10^{-4}$: $\|X^{(k)}-X^{(k-1)}\|_{H_\Delta}<10^{-4}\|X^{(k)}\|$, where we have defined $X^{(0)}$ as zero.

\subsection{Numerical example (1): $M =S^2$} 

\subsubsection{The tangent spaces, orthogonal projections and exponential maps of $S^2$} 
We regard the 2-sphere as $S^2 := \{(x_1, x_2, x_3) \in \RR^3 \mid x_1^2 + x_2^2 + x_3^2 = 1\}$.
Then the tangent spaces, their orthogonal projections and exponential maps in $S^2$ are given by the following explicit formulae: 
\[
    \begin{split}
        T_x S^2 &= \{ v \in \RR^3 \mid x^{\top}v = 0 \}, \\
        \pi_x(v) &= (I_3 - xx^{\top})v, \\
        \exp_x(v) &= \exp(vx^{\top} - xv^{\top})x,
    \end{split}
\]	
where $I_3$ denotes the identity matrix in $\RR^3$ and $\exp$ denotes the matrix exponential; here $x=(x_1,x_2,x_3)^{\top}$ is a column vector and $x^{\top}$ denotes its transpose.

\subsubsection{Euler angles} 
Vectors in $S^2$ have three parameters.
Euler angle representation is beneficial to reduce parameters of $S^2$.
Given $\gamma := (x, y, z) \in S^2$, its Euler angle representation is given  
 as follows: 
\begin{equation*}
    (x, y, z)  := (\sin \theta\sin\phi, \sin\theta\cos\phi, \cos\theta),
\end{equation*}
where $(\theta, \phi) \in [0, \pi) \times [0, 2\pi)$ are the Euler angles of $\gamma$ which are given by the formula
\begin{equation}
    (\theta, \phi) := \left(\arccos(z), \  \mathrm{sign}(x)\arccos\left(\frac{y}{\sqrt{x^2 + y^2}}\right) \right).
\end{equation}

\subsubsection{Counterexample to finite-time stopping phenomena} 
In~\cite{Giga2015}, an  example of constrained TV flow which does not reach the stationary point in finite time is shown.
Here is the statement.
\begin{theorem}[\cite{Giga2015}]
    Let $a,b \in S^2$ be two points
    represented by $a=(a_1,a_2,0)$ and $b=(a_1,-a_2,0)$
    for some $a_1,a_2 \in [-1,1]$
    with $a_1^2 + a_2^2 = 1$ and $a_1 > 0$.
    Take arbitrary $h_0 \in S^2 \cap \{x_2 = 0\}$ whose $x_3$-coordinate does not vanish.
    Then for any $L > 0$ and $0 < \ell_1 < \ell_2 < L$,
    the TV flow $u:[0, \infty) \to L^2((0,L); S^2)$ starting from the initial value
    \[
        u_0 = a 1_{ \left( 0 , \ell_1 \right) } %
        + h_0 1_{ \left( \ell_1, \ell_2 \right) } %
        + b 1 _ { \left( \ell_2 , L \right) }
    \]
    can be represented as
    \begin{equation}\label{eq:GigaKuroda}
        u(t) = a 1_{ \left( 0 , \ell_1 \right) } %
        + h(t) 1_{ \left( \ell_1, \ell_2 \right) } %
        + b 1 _ { \left( \ell_2 , L \right) }
    \end{equation}
    and $h(t)$ converges to $(1,0,0)$ as $t \to \infty$ but does not reach it in finite time.
\end{theorem}
In this theorem, $h(t)=(h_1(t),0,h_3(t))$ satisfies the following system of  differential equations: 
\begin{equation}\label{eq:ode_h}
    \frac{\mathrm{d}}{\mathrm{d}t} \left( h_1, h_3 \right) %
    = -\frac{\sqrt{2} a_1}{ c\sqrt{1-a_1 h_1}} \left(h_1^2-1 , h_1 h_3 \right),
\end{equation}
which  can be calculated numerically. 
Here $c=\ell_2-\ell_1$. 
Therefore we use it as a benchmark for the validation of our algorithm. 

\begin{remark}[Dirichlet problem] 
So far, in this paper, we have considered the Neumann problem of constrained TV flows, while this  example solves  the Dirichlet problem. 
Therefore, we can not apply the proposed scheme directly.
However, we can derive  the Dirichlet problem version of the proposed scheme just by replacing $T_{u_{\tau}^{(n)}}M_{\Delta}$ by
\begin{equation*}
    V(u_{\tau}^{(n)}) := \{ X \in T_{u_{\tau}^{(n)}}M_{\Delta} \mid X|_{\Omega_{\alpha}} = 0 \ \mbox{for} \ \alpha \in \partial \Delta \},
\end{equation*}
where
\begin{equation*}
    \partial \Delta := \{ \alpha \in \Delta \mid \mathcal{H}^{k-1}(\overline{\partial \Omega_{\alpha}} \cap \overline{\partial \Omega}) \not= 0 \}.
\end{equation*}
For more on  the Dirichlet problem, see \cite{Giga2005}.
\end{remark}

\subsubsection{Setup and numeral results} 
We use the following initial data $u_0$ with the Euler angles $\theta, \phi: \Omega= (0,1)\times (0,1) \to \RR$.
\begin{gather}
   \theta := \sum_{i=0}^{2}\bm{\theta}_{i}\mathbf{1}_{I_{i}}, \quad
    \phi := \sum_{i=0}^{2} \bm{\phi}_{i}\mathbf{1}_{I_{i}},
\end{gather}
where  
\begin{equation*}
  \bm{\theta} = \left( \frac{\pi}{2}, \frac{\pi}{4},  \frac{\pi}{2} \right), \quad
    \bm{\phi} = \left( \frac{\pi}{4},  \frac{\pi}{2},  \frac{3}{4} \pi \right),
\end{equation*}
and
\begin{gather*}
    I_0 = \left(0, \frac{2}{5}\right), \quad I_1 = \left(\frac{2}{5}, \frac{3}{5}\right), \quad I_2 = \left(\frac{3}{5}, 1\right).
\end{gather*}

We define the initial values  $a,b,h$ as
\[
    a = \left(\frac{1}{\sqrt{2}},\frac{1}{\sqrt{2}},0\right),\ %
    b = \left(\frac{1}{\sqrt{2}},-\frac{1}{\sqrt{2}},0\right),\ %
    h_0=\left(\frac{1}{\sqrt{2}},0,\frac{1}{\sqrt{2}}\right)
\]
and compare the results of our scheme with $\tau = 10^{-1}, 10^{-2}, 10^{-3}, 10^{-4}$.
We set the number of divisions in $[0,1]$ as 100 and used the explicit Euler method to solve the ordinary differential equation~\eqref{eq:ode_h} with the step size $10^{-6}$.

As we can see from Figure~\ref{fig:compared}, the behavior of the approximate solution computed by our proposed scheme and the one of the solution for \eqref{eq:ode_h} look similar. 
Figure~\ref{fig:logtaulogl2error} depicts the dependence of $\|u_\tau(t)-u(t)\|_{H_\Delta}$ on $\tau$ at time $t=0.2$ in $\log$-$\log$ scale. 
 We can see from this graph that the $L^2$ error decreases with the order $O(\tau)$ as $\tau$ tends to $0$ and it is faster than the $O(\sqrt{\tau})$-error estimate in Theorem~\ref{thm:mainerror}, which suggests that there is still room for improvement in our error estimate. 

\begin{figure}[btp]
    \begin{minipage}{.5\hsize}
		\includegraphics[width=\hsize]{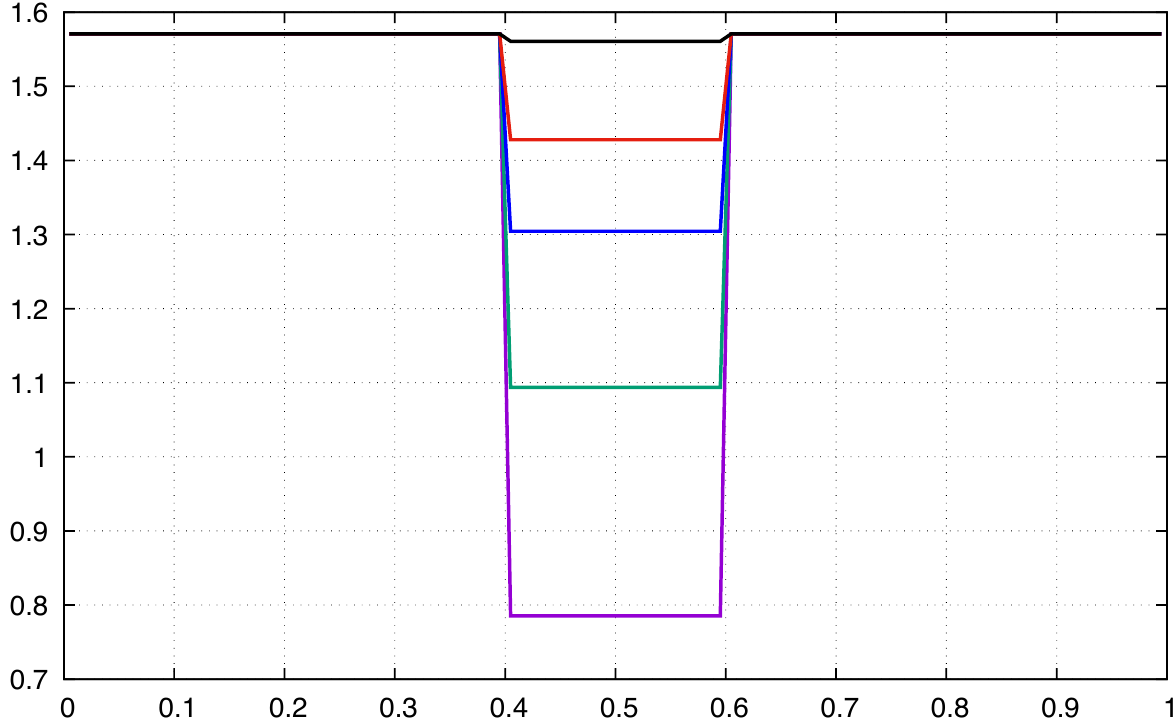}
    \end{minipage}%
    \begin{minipage}{.5\hsize}
    	\includegraphics[width=\hsize]{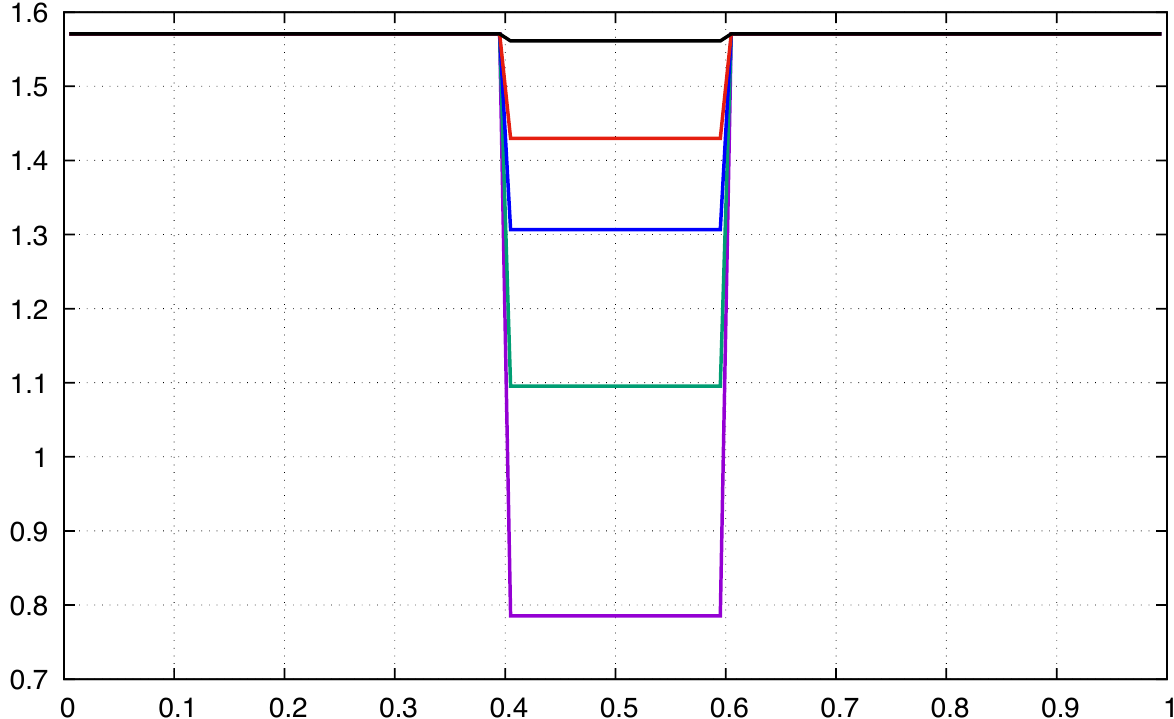}
    \end{minipage}%
    \caption{Comparison of flows at $t = 0.0,\ 0.07,\ 0.14,\ 0.21,\ 0.5$: The vertical axis represents the Euler angle $\theta$ of the flow. The left side is computed by our scheme, and the right side is computed by using explicit form in~\cite{Giga2015}.}
    \label{fig:compared}
\end{figure}
\begin{figure}[btp]
    \centering
    \includegraphics[width=.7\hsize]{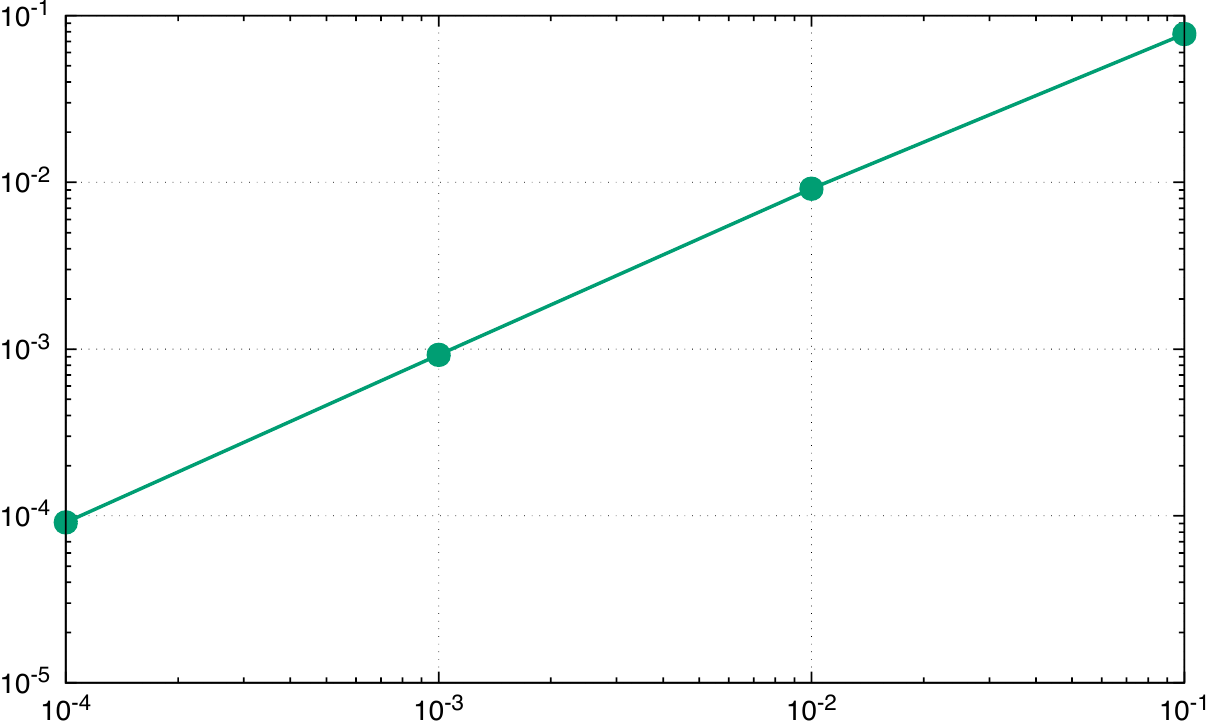}
    \caption{The $L^2$-error between our numerical scheme and the result in~\cite{Giga2015} at $t = 0.2$ with respect to the step size $\tau$.}
    \label{fig:logtaulogl2error}
\end{figure}

\subsection{Numerical example (2): $M = SO(3)$} 

\subsubsection{The tangent spaces, orthogonal projections and exponential maps of $SO(3)$} 

Let $M(3)$ denote the linear space of all three-by-three matrices. 
$SO(3)$ denotes the group of rotations, i.e., 
\begin{equation*}
    SO(3) := \{ x \in M(3) \mid x^{\top}x = xx^{\top} = I_3,\ \det x=1 \},
\end{equation*}
where $I_3$ denotes the identity matrix.
Then $SO(3)$ is regarded as a matrix Lie group in $M(3)$.
The set of all three-dimensional skew-symmetric matrices give the associated  Lie algebra to $SO(3)$:
\begin{equation*}
	so(3) := \{ X \in M(3)\mid X^{\top} = - X \}.
\end{equation*}
According to the general theory of Lie groups, 
$so(3)$ can be regarded as the tangent space $T_{I_3}(SO(3))$ at the identity, so  that
\begin{equation*}
	so_x(3) := T_x(SO(3)) = \{ xX \mid X \in so(3) \}.
\end{equation*}
We equip $M(3)$ with the inner product 
\[
	(X,Y) := \operatorname{trace}(X^{\top}Y) \quad\text{for}\quad X, Y \in M(3).
\]
Then it induces a Riemannian metric on the submanifold $SO(3)$, which is invariant by the left- and right-translation of $SO(3)$.
 The exponential map $\exp_e$ with respect to  
a bi-invariant Riemannian metric at the identity of a compact Lie group is given by the exponential map of a Lie algebra (\cite[Ch.\ I\!V, Theorem 3.3]{He}) or equivalently, that of a matrix, hence 
\[
	\exp_e(X) = \exp(X) \quad\text{for}\quad X \in T_e SO(3) \simeq so(3),
\]
where $e=I_3$ in this matrix group.
 Since the induced Riemannian metric is left-invariant, the exponential map at $x \in SO(3)$ is of the form
\begin{align*}
	\exp_x(X) &= L_x \exp_x(x^{-1}X) \\
	&= L_x \exp(x^{\top}X)
\end{align*}
for $X \in T_x\left(SO(3)\right) \simeq so_x(3)$, where $L_x$ denotes the left translation of $SO(3)$ given by $g \mapsto xg$.
 In other words, the exponential map  is given in the following simple form
\begin{equation*}
	\exp_x(X) = x\exp(x^{\top}X),  \quad X\in so_x(3).
\end{equation*}
 Since the decomposition
\begin{align*}
	&M(3) = so_x(3) + \left\{ xY \bigm| Y^{\top}=Y \right\},\\
	&X \mapsto \frac{X-xX^{\top}x}{2}+\frac{X+xX^{\top}x}{2}
\end{align*}
is a direct orthogonal decomposition of $M(3)$, for arbitrarily fixed $x \in SO(3)$, the orthogonal projection $\pi_x \colon M(3) \to so_x(3)$ is given by
\begin{equation}
	\pi_x(X) := \frac{X-xX^{\top}x}{2},  \quad  X \in M(3).
\end{equation}

\subsubsection{Euler angle} 
Rotations in $\RR^3$ have nine components.
The Euler angle (or Euler axis) representations are instrumental in reducing parameters of rotations.
Given $R := (R_{i,j})_{i,j=1}^3 \in SO(3)$, its Euler angle (or Euler axis) representation is given by Rodrigues' rotation formula
\begin{equation}
		R = \cos\theta I_3 + (1-\cos \theta)ee^{\top} + \sin \theta [e ]_x,
\end{equation}
where $\theta \in [0, 2\pi)$ and $e := (e_1, e_2, e_3) \in S^2$ denote the Euler angle and Euler  axis of $R$, respectively, and $[e]_x$ is the cross  product matrix of $e$.
The following formulae give them:
\begin{equation}
    \theta = \arccos{\left( \frac{R_{1,1} + R_{2,2} + R_{3,3} - 1}{2} \right)},
\end{equation}
\begin{equation}
    (e_1, e_2, e_3)  := \left( \frac{R_{3,2} - R_{2,3}}{2\sin \theta}, \quad  \frac{R_{1,3} - R_{3,1}}{2\sin \theta}, \quad \frac{R_{2,1} - R_{
1,2}}{2\sin \theta} \right),
\end{equation}
and
\[
  [e]_x = \left(
    \begin{array}{ccc}
      0 & -e_3 & e_2 \\
      e_3 & 0 & -e_1 \\
      -e_2 & e_1 & 0
    \end{array}
  \right).
\]
Since the Euler axis $e$ is in $S^2$, $e$ is represented by the spherical Euler angles $(\phi, \psi) \in [0,\pi)\times[0,2\pi)$.
Therefore, $R$ is represented by three parameters $(\theta, \phi, \psi) \in [0,2\pi) \times [0,\pi)\times[0,2\pi)$.

\subsubsection{Setup and numerical  results} 
We use the following initial data $u_0$ with the Euler angles $\theta, \phi, \psi :  \Omega= (0,1)\times (0,1) \to \RR$:
\begin{gather}
    \theta := \sum_{i, j=0}^{2}\bm{\theta}_{i,j}\mathbf{1}_{I_{i}\times J_{j}}, \quad
    \phi := \sum_{i, j=0}^{2} \bm{\phi}_{i,j}\mathbf{1}_{I_{i} \times J_{j}}, \quad
    \psi := \sum_{i, j=0}^{2} \bm{\psi}_{i,j}\mathbf{1}_{I_{i} \times J_{j}},
\end{gather}
where  
\begin{gather*}
  \bm{\theta} = \left(
    \begin{array}{ccc}
      0.35\pi & 0.2\pi & 0.55\pi \\
      0.81\pi & 0.64\pi & 0.4\pi \\
       0.1\pi & 0.7\pi & 0.3\pi
    \end{array}
  \right), \quad
    \bm{\phi} = \left(
    \begin{array}{ccc}
      0.4\pi & 0.5\pi & 0.7\pi \\
      0.5\pi & 0.3\pi & 0.4\pi \\
       0.6\pi & 0.3\pi & 0.4\pi
    \end{array}
  \right), \\
  \quad
    \bm{\psi} = \left(
    \begin{array}{ccc}
      0.2\pi & 0.25\pi & 0.3\pi \\
      0.25\pi & 0.225\pi & 0.2\pi \\
       0.3\pi & 0.2\pi & 0.35\pi
    \end{array}
  \right),
\end{gather*}
and
\begin{align*}
    &I_0 = \left(0, \frac{2}{5}\right), \quad I_1 = \left(\frac{2}{5}, \frac{3}{5}\right), \quad I_2 = \left(\frac{3}{5}, 1\right), \\
    &J_0 = \left(0, \frac{1}{5}\right), \quad J_1 = \left(\frac{1}{5}, \frac{4}{5}\right), \quad J_2 = \left(\frac{4}{5}, 1\right).
\end{align*}

Figures~\ref{fig:SO3_1} and \ref{fig:SO3_2} depict the results of numerical experiments under the above setting at times  $t=0.0, 0.05, 0.1, 0.25$ and $t=0.0, 0.5, 0.75, 1.0$, respectively, where $\Omega$ is divided as $\bigcup_{i=1}^{N_x}\bigcup_{j=1}^{N_y}\Omega_{i,j}$, $\Omega_{i,j}:=((i-1)\Delta x,i\Delta x)\times((j-1)\Delta y,j\Delta y)$, in which we have defined both of $N_x$ and $N_y$ as $25$ and $\Delta x=1/N_x$ and $\Delta y=1/N_y$.
Although no  benchmark test for $SO(3)$-valued TV flow is known,  we can see that our proposed numerical scheme works well since the facet-preserving property is satisfied, and the numerical solution finally reaches the constant solution.

\appendix
\section{About the  constant $C_M$}\label{sec:the_constant} 

In this section, we derive  a bound of the constant $C_M$ defined  in \eqref{eq:Curv+LSF}.
We recall several notations developed in computational geometry.
A point $x \in \RR^{\ell}$ is said to have the unique nearest point if there exists a unique point $p(x) \in M$ such that $p(x) \in \argmin_{p \in M}\norm{x - p}_{\RR^{\ell}}$.
Let $S_0(M)$ denote the set of all points in $\RR^{\ell}$ which \textit{do not} have the unique nearest point.
The closure $S(M)$ of $S_0(M)$ is called the medial axis of ${M}$.
The local feature size $\lfs(M)$ of ${M}$ is the quantity defined by
\begin{equation*} 
	\lfs({M}) := \inf_{p \in {M}}\inf_{q \in S_0(M)}\norm{p - q}_{\RR^{\ell}}.
\end{equation*}
Now, we assume that $M$ is compact. 
Then, $\lfs({M})$ is positive because $M$ has positive reach (\cite{Foote1984}).
We use the quantity $\mathrm{lfs}(M)$ to obtain that 
\begin{equation}\label{eq:LFS}
	\norm{p - q}_{\RR^{\ell}} \leq \Dist_{{M}}(p,q) 
	\leq 2\max \left\{ 1, \frac{ \mathrm{Diam}({M})}{\lfs({M})} \right\} \norm{p - q}_{\RR^{\ell}}
\end{equation}
for each point $p$ and $q$ in $M$.
Here, $\mathrm{Dist}_M(p,q)$ denotes the geodesic distance between $p$ and $q$.
On the other hand, assuming that $M$ is path-connected, we have
\begin{equation}\label{eq:Curv}
    \| \pi^{\bot}_p(p - q) \|_{\RR^{\ell}} \leq \frac{1}{2} \ \mathrm{Curv}(M) \  \mathrm{Dist}_M(p,q)^2
\end{equation}
for all $p, q \in M$.
Therefore, if $M$ is a path-connected and a compact submanifold of $\RR^{\ell}$, then we combine  the inequalities \eqref{eq:Curv} and \eqref{eq:LFS} to obtain 
\begin{equation}\label{eq:Curv+LFS}
	\norm{\pi^{\bot}_p(p-q)}_{\RR^{\ell}}
    \leq  2 \ \mathrm{Curv}({{M}}) \ \max \left\{ 1, \frac{ \mathrm{Diam}({M})}{\lfs({M})} \right\}^2 \norm{p - q}_{\RR^{\ell}}^2
\end{equation}
for all $p, q \in M$.
Hence, we have 
\begin{equation}
    C_M \leq 2 \ \mathrm{Curv}({{M}}) \ \max \left\{ 1, \frac{ \mathrm{Diam}({M})}{\lfs({M})} \right\}^2.
\end{equation}
Finally, we remark that the proofs of \eqref{eq:LFS} and \eqref{eq:Curv} are found in \cite{Taguchi2018}.

\begin{figure}[p]
    \centering
    \includegraphics[width=1.0\hsize]{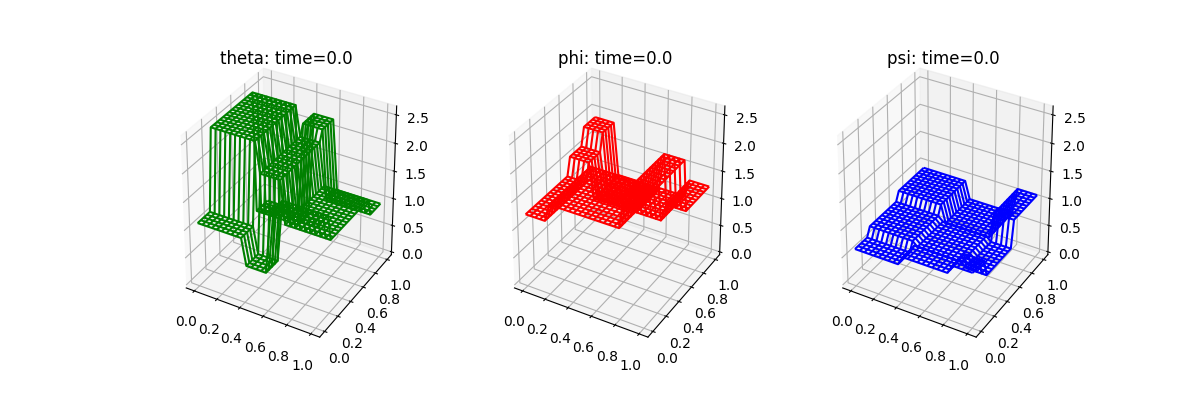}
    \includegraphics[width=1.0\hsize]{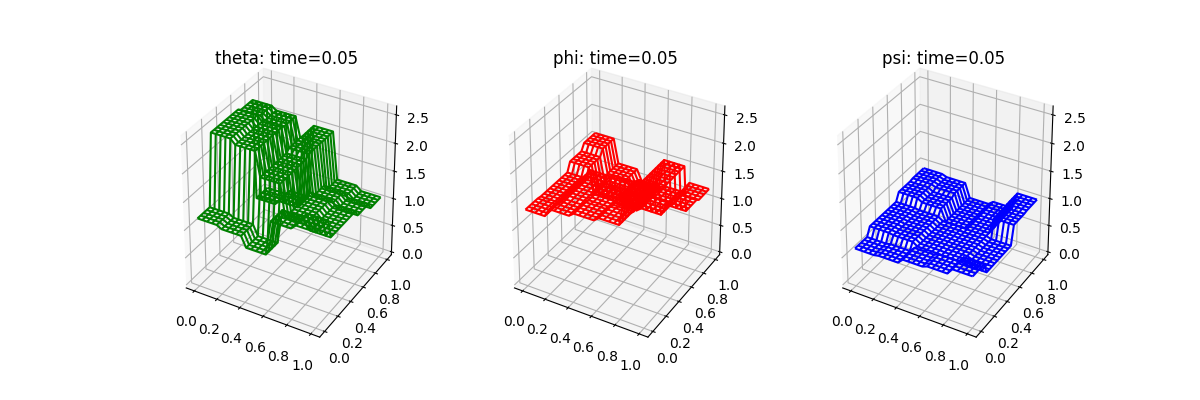}
    \includegraphics[width=1.0\hsize]{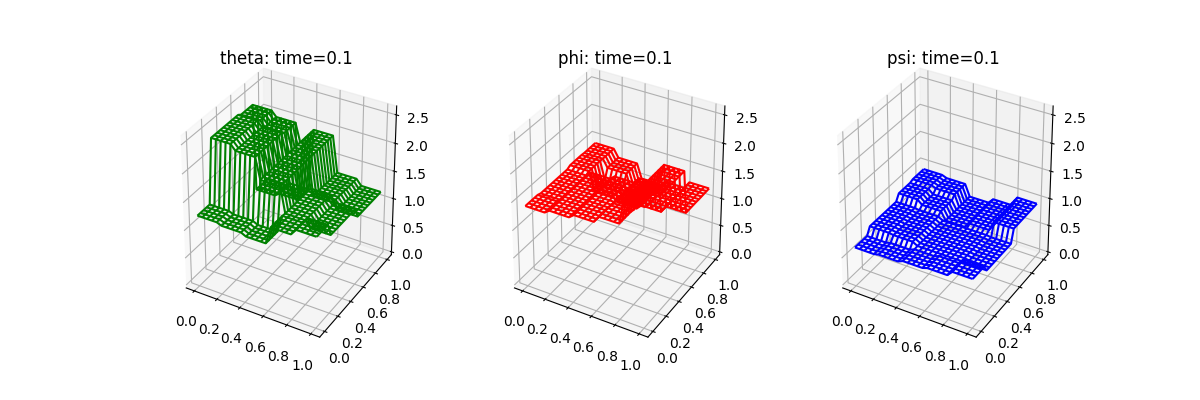}
    \includegraphics[width=1.0\hsize]{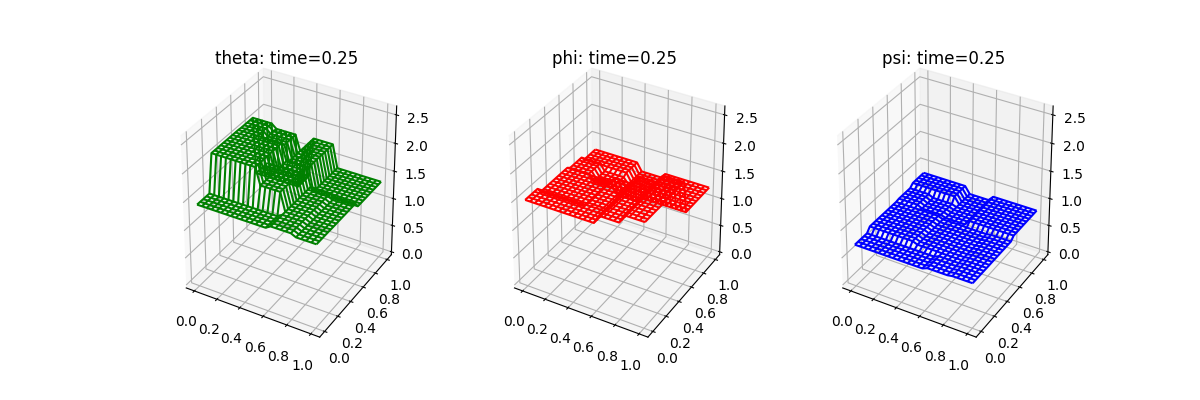}
    \caption{Numerical simulation of $SO(3)$-valued TV flow at $t = 0.0, 0.05, 0.1, 0.25$.}
    \label{fig:SO3_1}
\end{figure}
\begin{figure}[p]
    \centering
    \includegraphics[width=1.0\hsize]{SO3_0.png}
    \includegraphics[width=1.0\hsize]{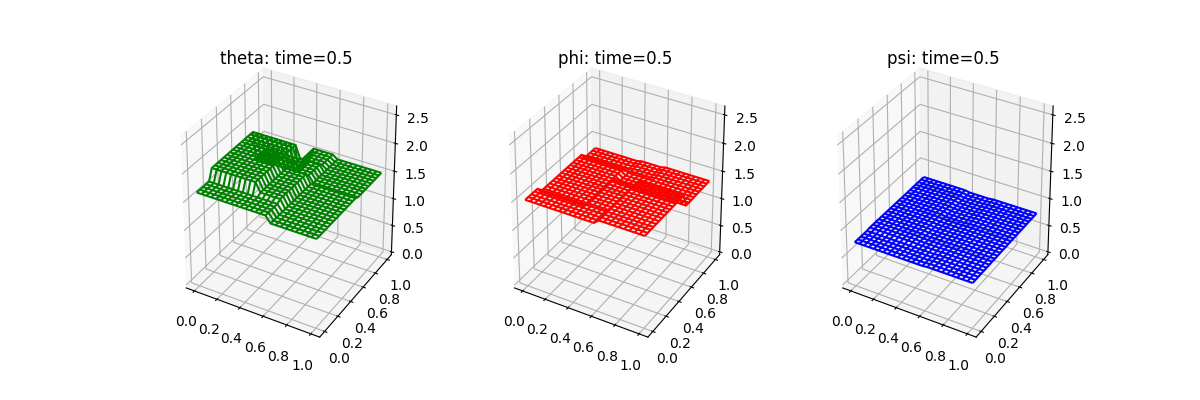}
    \includegraphics[width=1.0\hsize]{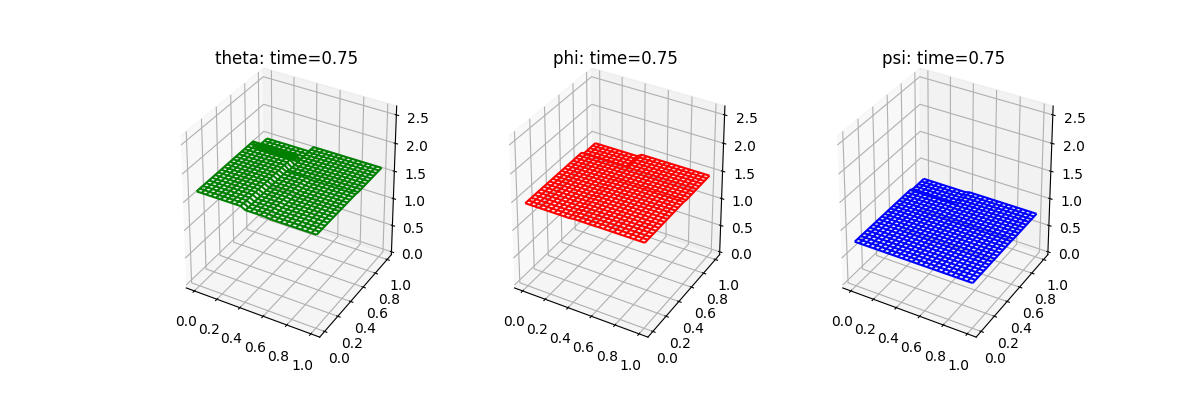}
    \includegraphics[width=1.0\hsize]{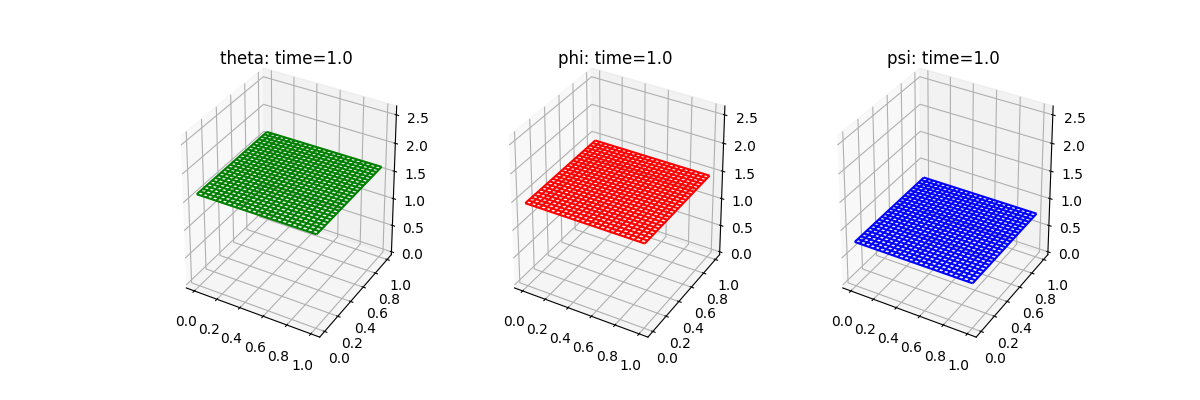}
    \caption{Numerical simulation of $SO(3)$-valued TV flow at $t = 0.0, 0.5, 0.75, 1.0$.}
    \label{fig:SO3_2}
\end{figure}


\end{document}